\documentclass[12pt]{article}

\usepackage{times}

\usepackage{amsmath}

\usepackage[a4paper,text={128mm,185mm}, centering]{geometry}

\usepackage{changepage}

\usepackage{titlesec}

\titleformat{\section}[hang]%
{\bfseries\large}{\thesection.}{1ex}{}%

\titleformat{\subsection}[hang]%
{\bfseries}{\thesubsection}{1ex}{}%

\usepackage{amsthm}
\usepackage{fancyhdr}
\theoremstyle{plain}
\newtheorem{theorem}{Theorem}[section]
\newtheorem{lem}[theorem]{Lemma}
\newtheorem{prop}[theorem]{Proposition}
\newtheorem{cor}[theorem]{Corollary}
\newtheorem{definition}[theorem]{Definition} 

\theoremstyle{definition}

\newtheorem{prob}[theorem]{Problem}
\newtheorem{notation}[theorem]{Notation}
\newtheorem{obs}[theorem]{Observation}

\usepackage[matrix,arrow,curve,cmtip]{xy}
\usepackage{amssymb}
\usepackage{latexsym}
\usepackage{graphicx}
\usepackage{tikz}
\usetikzlibrary{matrix,arrows,decorations.pathmorphing}

\title{\vskip 5pt  \bf  GLOBULARLY GENERATED DOUBLE CATEGORIES II: THE CANONICAL DOUBLE PROJECTION}
\author{\itshape\bfseries { Juan Orendain}}
\date{}

\begin{document}
	\maketitle
	

	\vskip 25pt
	\begin{adjustwidth}{0.5cm}{0.5cm}
		{\small
			
			{\bf Résumé.} Il s'agit du deuxième volet d'une série d'articles en deux parties portant sur les catégories doubles librement globulairement engendrées. Nous introduisons la construction canonique de la double projection. Celle-ci transporte l'information des catégories doubles librement globulairement engendrées aux catégories doubles  définies par le même ensemble de données globulaires et verticales. Nous utilisons cette double projection pour définir des extensions fonctorielles linéaires formelles compatibles de la forme standard de Haagerup et de l'opération de fusion de Connes aux morphismes entre facteurs  d'index éventuellement infini. Nous l'utilisons encore pour montrer que la construction de la double catégorie librement globulairement engendrée est adjointe à gauche à l' "horizontalisation décorée". Nous interprétons ainsi les catégories doubles librement globulairement engendrées comme des analogues formellement  décorés des catégories doubles de quintettes et comme des générateurs pour l'internalisation.

			{\bf Abstract.} This is the second installment of a two part series of papers studying free globularly generated double categories. We introduce the canonical double projection construction. The canonical double projection translates information from free globularly generated double categories to double categories defined through the same set of globular and vertical data. We use the canonical double projection to define compatible formal linear functorial extensions of the Haagerup standard form and the Connes fusion operation to possibly-infinite index morphisms between factors. We use the canonical double projection to prove that the free globularly generated double category construction is left adjoint to decorated horizontalization. We thus interpret free globularly generated double categories as formal decorated analogs of double categories of quintets and as generators for internalizations.
		}
	\end{adjustwidth}
	
	
	\section{Introduction}

	\noindent  Globularly generated double categories were introduced by the author in \cite{yo1} in order to study ways of minimally lifting bicategories into double categories along possible categories of vertical arrows. Free globularly generated double categories were later introduced in \cite{yo2}. The free globularly generated double category construction minimally associates to every bicategory together with a possible category of vertical arrows, a double category fixing this set of initial data. Free globularly generated double categories are related to free products of groups and monoids, free double categories in the sense of \cite{DawsonPareFree} and to the Ehresmann double category of quintets construction \cite{Ehr3}, they define numerical invariants for both bicategories and double categories, and provide formal linear functorial extensions of operations in the representation theory of von Neumann algebras.

	In this paper we study the canonical projection double functor. The canonical double projection transfers information from free globularly generated double categories to other double categories defined through the same set of initial data. In the language of \cite{yo1,yo2} given a decorated bicategory $(\mathcal{B},\mathcal{B}^*)$, i.e. given a bicategory $\mathcal{B}$ together with a category $\mathcal{B}^*$ having the same set of objects as $\mathcal{B}$, and a globularly generated double category $C$ internalizing $(\mathcal{B},\mathcal{B}^*)$, i.e. having $\mathcal{B}^*$ as category of objects and $\mathcal{B}$ as horizontal bicategory, the canonical double projection associated to $C$ is a strict double functor
	
	\[\pi^C:Q_{(\mathcal{B},\mathcal{B}^*)}\to C\]

	\noindent from the free globularly generated double category $Q_{(\mathcal{B},\mathcal{B}^*)}$ associated to $(\mathcal{B},\mathcal{B}^*)$, to $C$, such that $\pi^C$ is surjective on squares and acts as the identity on objects, vertical morphisms, horizontal morphisms and 2-cells of $\mathcal{B}$. We summarize this by saying that the restriction of the decorated horizontalization pseudofunctor $H^*\pi^C$ of $\pi^C$, see \cite[Section 2.6]{yo1}, to $(\mathcal{B},\mathcal{B}^*)$, is the identity on $(\mathcal{B},\mathcal{B}^*)$, or equivalently by the equation:
	
	\[H^*\pi^C\restriction_{\mathcal{B}}=id_\mathcal{B}\]

	\noindent  In Theorem \ref{thmprojection1} we prove canonical double projections always exist and that are uniquely determined by the above properties. We interpret the properties defining canonical double projections by considering free globularly generated double categories and canonical double projections as generators and relations presentations of general globularly generated double categories. In Section 3 we exploit this to provide bounds for numerical invariants of double categories, to prove that every lift of a decorated 2-groupoid canonically contains a double groupoid, and to provide compatible formal linear functorial extensions of the Haagerup standard form and the Connes fusion operation extending the corresponding functors provided in \cite{Bartels1}. In Section 4 we extend the free globularly generated double category construction to a functor $Q:\mbox{\textbf{bCat}}^*\to \mbox{\textbf{dCat}}$ and in Theorem \ref{thmadjoint} we prove that $Q$ fits into a left adjoint pair $(Q,H^*)$ with the collection of canonical double projections as counit thus making free globularly generated double categories free objects with respect to $H^*$, see Corollary \ref{Qfree}. We regard this result as a fibered version of the classic result of \cite{Spencer} and \cite{BrownMosa} exchanging horizontalization $H$ with decorated horizontalization $H^*$ and the Ehresmann double category of quintets functor \textbf{Q} with $Q$. We provide a more detailed account of the contents and motivation for the main results of the paper.

	\

	\noindent \textit{Internalization}

	\

	\noindent Given a bicategory $\mathcal{B}$ we will say that a category $\mathcal{B}^*$ is a decoration for $\mathcal{B}$ if the collection of 0-cells of $\mathcal{B}$ and the collection of objects of $\mathcal{B}^*$ are equal. In that case we say that the pair $(\mathcal{B}^*,\mathcal{B})$ is a decorated bicategory. Given a double category $C$ the pair $(C_0,HC)$ formed by the category of objects and the horizontal bicategory of $C$ is a decorated bicategory. We will write $H^*C$ for this decorated bicategory. We will call $H^*C$ the decorated horizontalization of $C$. We are interested in the question of how generic the decorated horizontalization construction is, i.e. we are interested in how and when a given decorated bicategory con be presented as the decorated horizontalization of a double category. We study solutions to the following problem:

	\begin{prob}\label{prob}
		Let $(\mathcal{B}^*,\mathcal{B})$ be a decorated bicategory. Find double categories $C$ satisfying the equation $H^*C=(\mathcal{B}^*,\mathcal{B})$.
	\end{prob}

	\noindent We call any solution $C$ to the equation $H^*C=(\mathcal{B}^*,\mathcal{B})$ an internalization of $(\mathcal{B}^*,\mathcal{B})$. Problem \ref{prob} admits the following pictorial interpretation: Suppose we are given a collection of globular diagrams of the form:
	
	
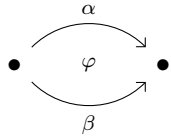
\begin{center}
		
		\begin{tikzpicture}
			\matrix(m)[matrix of math nodes, row sep=4em, column sep=4em,text height=1.5ex, text depth=0.25ex]
			{\bullet&\bullet\\};
			\path[->,font=\scriptsize,>=angle 90]
			(m-1-1) edge [bend right=45]node [below]{$\beta$} (m-1-2)
			edge [bend left=45] node [above]{$\alpha$} (m-1-2)
			edge [white]node[black][fill=white]{$\varphi$} (m-1-2)

			;
		\end{tikzpicture}
	\end{center}

	\noindent forming a bicategory, together with a collection of vertical arrows of the form:

	
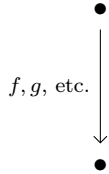
\begin{center}
		
		\begin{tikzpicture}
			\matrix(m)[matrix of math nodes, row sep=4em, column sep=4em,text height=1.5ex, text depth=0.25ex]
			{\bullet\\
				\bullet\\};
			\path[->,font=\scriptsize,>=angle 90]
			(m-1-1) edge node [left]{$f,g$, etc.} (m-2-1)

			;
		\end{tikzpicture}
	\end{center}

	\noindent forming a category, satisfying the condition that the collection of vertices of both sets of diagrams coincide. With this data we can form hollow squares of the form:

	
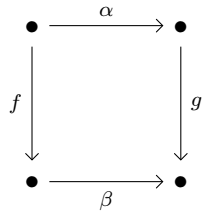
\begin{center}
		
		\begin{tikzpicture}
			\matrix(m)[matrix of math nodes, row sep=4em, column sep=4em,text height=1.5ex, text depth=0.25ex]
			{\bullet&\bullet\\
				\bullet&\bullet\\};
			\path[->,font=\scriptsize,>=angle 90]
			(m-1-1) edge node [above]{$\alpha$} (m-1-2)
			edge node [left]{$f$} (m-2-1)
			(m-1-2) edge node [right]{$g$}(m-2-2)
			(m-2-1) edge node [below]{$\beta$}(m-2-2)
			
			(m-1-1) edge [white]node[black][fill=white]{} (m-2-2)

			;
		\end{tikzpicture}
	\end{center}

	\noindent formed by the edges of the diagrams we are provided with. Problem \ref{prob} asks about ways to fill these hollow squares \textit{equivariantly} with respect to the globular diagrams in our set of initial conditions. That is, Problem \ref{prob} asks for the existence of systems of solid squares of the form:

	
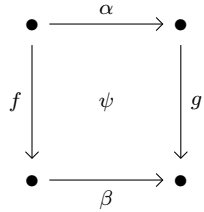
\begin{center}
		
		\begin{tikzpicture}
			\matrix(m)[matrix of math nodes, row sep=4em, column sep=4em,text height=1.5ex, text depth=0.25ex]
			{\bullet&\bullet\\
				\bullet&\bullet\\};
			\path[->,font=\scriptsize,>=angle 90]
			(m-1-1) edge node [above]{$\alpha$} (m-1-2)
			edge node [left]{$f$} (m-2-1)
			(m-1-2) edge node [right]{$g$}(m-2-2)
			(m-2-1) edge node [below]{$\beta$}(m-2-2)
			
			(m-1-1) edge [white]node[black][fill=white]{$\psi$} (m-2-2)

			;
		\end{tikzpicture}
	\end{center}

	\noindent forming a double category such that every square as above admits an interpretation as a globular diagram together with extra structure provided only by our category of vertical arrows, that is such that the only solid squares of the form:

	
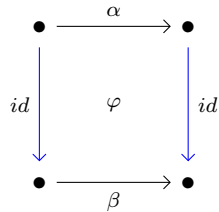
\begin{center}
		
		\begin{tikzpicture}
			\matrix(m)[matrix of math nodes, row sep=4em, column sep=4em,text height=1.5ex, text depth=0.25ex]
			{\bullet&\bullet\\
				\bullet&\bullet\\};
			\path[->,font=\scriptsize,>=angle 90]
			(m-1-1) edge node [above]{$\alpha$} (m-1-2)
			edge [blue]node [black,left]{$id$} (m-2-1)
			(m-1-2) edge [blue]node [black,right]{$id$}(m-2-2)
			(m-2-1) edge node [below]{$\beta$}(m-2-2)
			
			(m-1-1) edge [white]node[black][fill=white]{$\varphi$} (m-2-2)

			;
		\end{tikzpicture}
	\end{center}

	\noindent are the globular diagrams provided as set of initial conditions. We regard the decorated horizontalization condition of Problem \ref{prob} a formalization of the equivariance condition on the above squares. 
	
	Constructions of this sort appear in different parts of the theory of double categories. Notably the double category of squares and the double category of commuting squares construction, the Ehresmann double category of quintets construction \cite{Ehr3}, the double category of adjoint pairs construction \cite{Palmquist}, and the double categories of spans and cospans constructions all follow the pattern described above. Double categories of squares have categories as globular and vertical sets of initial data, the double category of quintets has a given 2-category and the corresponding category of 1-cells as set of initial data, the double category of adjoints has a given 2-category together with adjoint pairs of 1-cells as set of initial data, and the double category of spans/cospans has the bicategory of spans/cospans of a category with pushouts/pullbacks and the arrows of this category as globular and vertical sets of initial data. In all cases solid squares are carefully chosen so as to encode different aspects of the globular theory. 
	
	Our main interest in Problem \ref{prob} comes from the theory of representations of von Neumann algebras. In \cite{Bartels1,Bartels2} a double category of semisimple von Neumann algebras, Hilbert bimodules and finite index bounded equivariant intertwiners was defined. See \cite{Bartels4} for applications to conformal field theory and the Stolz-Teichner program. The main goal of this construction is to serve as an intermediate step in the construction of an internal bicategory of coordinate free conformal nets. The main obstruction for the existence of an internal bicategory of general, i.e. not-necessarily-semisimple coordinate free conformal nets, is the existence of a compatible pair of tensor functors extending the Haagerup standard form construction \cite{Haagerup} and the Connes fusion operation to not-necessarily-finite index morphisms of semisimple von Neumann algebras. The existence of such tensor functors is equivalent to the existence of a tensor double category of (not-necessarily-semisimple) von Neumann algebras, Hilbert bimodules, and (not-necessarily-finite index) equivariant intertwiners extending the double category defined in \cite{Bartels1}. We achieve this in this paper in the case of linear double categories of factors.

	\

	\noindent \textit{Globularly generated double categories}
	
	\

	\noindent Globularly generated double categories were introduced in \cite{yo1} as minimal solutions to Problem \ref{prob}. A double category $C$ is globularly generated if $C$ is generated by its collection of globular squares. Pictorially a double category $C$ is globularly generated if every square of $C$ can be written as vertical and horizontal compositions of squares of the form:
	
	
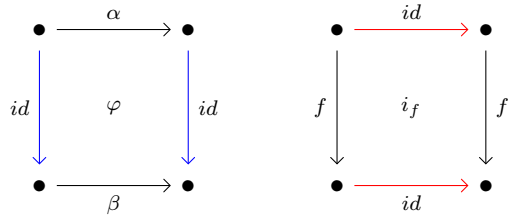
\begin{center}
		
		\begin{tikzpicture}
			\matrix(m)[matrix of math nodes, row sep=4em, column sep=4em,text height=1.5ex, text depth=0.25ex]
			{\bullet&\bullet&\bullet&\bullet\\
				\bullet&\bullet&\bullet&\bullet\\};
			\path[->,font=\scriptsize,>=angle 90]
			(m-1-1) edge node [above]{$\alpha$} (m-1-2)
			edge [blue]node [black,left]{$id$} (m-2-1)
			(m-1-2) edge [blue]node [black,right]{$id$}(m-2-2)
			(m-2-1) edge node [below]{$\beta$}(m-2-2)
			
			(m-1-1) edge [white]node[black][fill=white]{$\varphi$} (m-2-2)

			(m-1-3) edge [red]node [black,above]{$id$} (m-1-4)
			edge node [left]{$f$} (m-2-3)
			(m-1-4) edge node [right]{$f$}(m-2-4)
			(m-2-3) edge [red]node [black,below]{$id$}(m-2-4)
			
			(m-1-3) edge [white]node[black][fill=white]{$i_f$} (m-2-4)

			;
		\end{tikzpicture}
	\end{center}

	\noindent Given a double category $C$ we write $\gamma C$ for the sub-double category of $C$ generated by squares of the above form. We call $\gamma$ the globularly generated piece of $C$. $\gamma C$ is globularly generated, satisfies the equation 
	
	\[H^*C=H^*\gamma C\]
	
	\noindent and is contained in every sub-double category $D$ of $C$ satisfying the equation $H^*C=H^*D$. Moreover, a double category $C$ is globularly generated if and only if $C$ does not contain proper sub-double categories satisfying the above equation. Globularly generated double categories are thus minimal with respect to $H^*$.
	
	The comments in the previous paragraph admit the following categorical interpretation: Let \textbf{dCat}, \textbf{gCat} and \textbf{bCat}$^*$ denote the category of double categories and double functors, the full sub-category of \textbf{dCat} generated by globularly generated double categories and the category of decorated bicategories and decorated pseudofunctors respectively. Decorated horizontalization extends to a functor $H^*:\mbox{\textbf{dCat}}\to\mbox{\textbf{bCat}}^*$ and the globularly generated piece construction extends to a functor $\gamma:\mbox{\textbf{dCat}}\to\mbox{\textbf{gCat}}$. In \cite[Proposition 3.6]{yo1} it is proven that $\gamma$ is a coreflector of \textbf{gCat} in \textbf{dCat}. It is easily seen that this implies that $\gamma$ is a Grothendieck fibration. Moreover, $H^*$ is constant on $\gamma$-fibers. We present this through the following diagram:

	
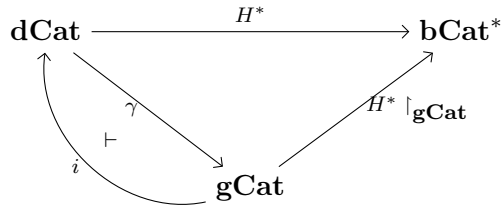
\begin{center}
		
		\begin{tikzpicture}
			\matrix(m)[matrix of math nodes, row sep=4em, column sep=4em,text height=1.5ex, text depth=0.25ex]
			{\mbox{\textbf{dCat}}&&\mbox{\textbf{bCat}}^*\\
				&\mbox{\textbf{gCat}}&\\};
			\path[->,font=\scriptsize,>=angle 90]
			(m-1-1) edge node [above]{$H^*$} (m-1-3)
			edge node [left]{$\gamma$} (m-2-2)
			(m-2-2) edge node [right]{$H^*\restriction_{\mbox{\textbf{gCat}}}$}(m-1-3)
			(m-2-2) edge [bend left=55] node [black,left]{$i$}(m-1-1)
			(m-2-2) edge [white,bend left=30] node [black, fill=white]{$\vdash$}(m-1-1)

			;
		\end{tikzpicture}
	\end{center}

	\noindent where $i$ denotes the inclusion of \textbf{gCat} in \textbf{dCat}. The above diagram breaks Problem \ref{prob} into the problem of studying bases of $\gamma$ and then understanding the double categories in each fiber. We follow this strategy and thus study globularly generated double categories, i.e. bases with respect to $\gamma$.

	\
	
	\noindent \textit{The vertical filtration}
	
	\

	\noindent Globularly generated double categories admit a helpful combinatorial description provided in the form of a filtration of their categories of squares. Given a globularly generated double category $C$ we write $V^1_C$ for the category formed by vertical compositions of squares of the form:

	
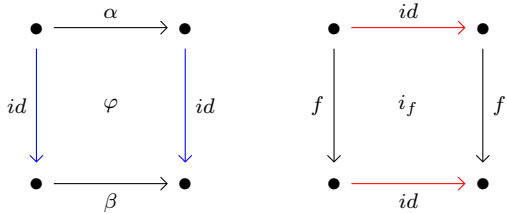
\begin{center}
		
		\begin{tikzpicture}
			\matrix(m)[matrix of math nodes, row sep=4em, column sep=4em,text height=1.5ex, text depth=0.25ex]
			{\bullet&\bullet&\bullet&\bullet\\
				\bullet&\bullet&\bullet&\bullet\\};
			\path[->,font=\scriptsize,>=angle 90]
			(m-1-1) edge node [above]{$\alpha$} (m-1-2)
			edge [blue]node [black,left]{$id$} (m-2-1)
			(m-1-2) edge [blue]node [black,right]{$id$}(m-2-2)
			(m-2-1) edge node [below]{$\beta$}(m-2-2)
			
			(m-1-1) edge [white]node[black][fill=white]{$\varphi$} (m-2-2)

			(m-1-3) edge [red]node [black,above]{$id$} (m-1-4)
			edge node [left]{$f$} (m-2-3)
			(m-1-4) edge node [right]{$f$}(m-2-4)
			(m-2-3) edge [red]node [black,below]{$id$}(m-2-4)
			
			(m-1-3) edge [white]node[black][fill=white]{$i_f$} (m-2-4)

			;
		\end{tikzpicture}
	\end{center}
	
	\noindent and we denote by $H^1_C$ the (possibly weak) category formed by horizontal compositions of squares of this form. Assuming we have defined $V^k_C$ and $H^k_C$ through vertical and horizontal compositions respectively, we make $V^{k+1}_C$ to be the category generated by squares in $H^k_C$ and $H^{k+1}_C$ the (possibly weak) category generated by squares in $V^{k+1}_C$. The category of squares $C_1$ of $C$ satisfies the equation $C_1=\varinjlim V^k_C$. We define the length $\ell C\in\mathbb{N}\cup\left\{\infty\right\}$ of a double category $C$ as the minimal $k$ such that the equation $\gamma C_1=V^k_{\gamma C}$ holds. Intuitively the vertical length of a double category $C$ measures the complexity of expressions of squares in $C$ by globular and horizontal identity squares.
	
	We further explain the vertical filtration construction through the following pictorial representation: We regard the globular and horizontal identity squares of a double category $C$ as the simplest possible squares of $C$, i.e. we regard these squares as having 'complexity' 0. We thus represent globular and horizontal identity squares diagramatically as squares marked by 0, i.e. as:

	\begin{center}

		\tikzset{every picture/.style={line width=0.75pt}} 
		
		\begin{tikzpicture}[x=0.75pt,y=0.75pt,yscale=-1,xscale=1]
			
			\draw   (292.77,151.6) -- (369.82,151.6) -- (369.82,228.65) -- (292.77,228.65) -- cycle ;
			
			\draw (324.95,185.52) node [anchor=north west][inner sep=0.75pt]  [font=\scriptsize,xscale=0.9,yscale=0.9]  {$0$};

		\end{tikzpicture}

	\end{center}

	\noindent The collection of such squares is what in Section 2 we denote by $\mathbb{G}$. Observe that the collection of 0-marked squares is closed under horizontal composition. Squares in $V^1_C$ are those squares in $C$ admitting a subdivision as vertical composition of 0-marked squares. Diagrammatically every square in $V^1_C$ admits a decomposition as:
	
	\begin{center}

		\tikzset{every picture/.style={line width=0.75pt}} 
		
		\begin{tikzpicture}[x=0.75pt,y=0.75pt,yscale=-1,xscale=1]
			
			\draw   (290.34,171.28) -- (369.16,171.28) -- (369.16,250.09) -- (290.34,250.09) -- cycle ;
			\draw    (290.41,190.71) -- (369.3,190.41) ;
			\draw    (290.64,210.71) -- (368.86,210.66) ;
			\draw    (290.19,230.49) -- (369.3,230.41) ;
			
			\draw (324.45,177.54) node [anchor=north west][inner sep=0.75pt]  [font=\scriptsize,xscale=0.9,yscale=0.9]  {$0$};
			\draw (324.34,195.76) node [anchor=north west][inner sep=0.75pt]  [font=\scriptsize,xscale=0.9,yscale=0.9]  {$0$};
			\draw (323.9,236.43) node [anchor=north west][inner sep=0.75pt]  [font=\scriptsize,xscale=0.9,yscale=0.9]  {$0$};
			\draw (333.41,211.46) node [anchor=north west][inner sep=0.75pt]  [font=\scriptsize,rotate=-90.3,xscale=0.9,yscale=0.9]  {$\cdots $};

		\end{tikzpicture}

	\end{center}
	
	\noindent where we draw internal 0-marked squares as rectangles for convenience. If a square as above is not globular or a horizontal identity, i.e. is not 0-marked, we mark it with 1. We represent 1-marked squares pictorially as:
	
	\begin{center}

		\tikzset{every picture/.style={line width=0.75pt}} 
		
		\begin{tikzpicture}[x=0.75pt,y=0.75pt,yscale=-1,xscale=1]
			
			\draw   (292.77,169.6) -- (369.82,169.6) -- (369.82,246.65) -- (292.77,246.65) -- cycle ;
			
			\draw (324.95,203.52) node [anchor=north west][inner sep=0.75pt]  [font=\scriptsize,xscale=0.9,yscale=0.9]  {$1$};

		\end{tikzpicture}

	\end{center}
	
	\noindent Squares in $H^1_C$ are thus those squares in $C$ that admit a subdivision as horizontal composition of squares marked with $i\leq 1$. Given two horizontally composable squares $\varphi,\psi$ in $V^1_C$ we might be able to find compatible vertical subdivisions of $\varphi$ and $\psi$ in 0-marked squares, i.e. we might be able to represent the horizontal composition of $\varphi$ and $\psi$ as:
	
	\begin{center}

		\tikzset{every picture/.style={line width=0.75pt}} 
		
		\begin{tikzpicture}[x=0.75pt,y=0.75pt,yscale=-1,xscale=1]
			
			\draw   (251.59,141) -- (330.41,141) -- (330.41,219.82) -- (251.59,219.82) -- cycle ;
			\draw    (251.66,160.43) -- (330.55,160.14) ;
			\draw    (251.89,180.43) -- (330.11,180.38) ;
			\draw    (251.44,200.21) -- (330.55,200.14) ;
			\draw   (331.09,141) -- (409.91,141) -- (409.91,219.82) -- (331.09,219.82) -- cycle ;
			\draw    (331.16,160.43) -- (410.05,160.14) ;
			\draw    (331.39,180.43) -- (409.61,180.38) ;
			\draw    (330.94,200.21) -- (410.05,200.14) ;
			
			\draw (285.7,147.26) node [anchor=north west][inner sep=0.75pt]  [font=\scriptsize,xscale=0.9,yscale=0.9]  {$0$};
			\draw (285.59,165.48) node [anchor=north west][inner sep=0.75pt]  [font=\scriptsize,xscale=0.9,yscale=0.9]  {$0$};
			\draw (285.15,206.15) node [anchor=north west][inner sep=0.75pt]  [font=\scriptsize,xscale=0.9,yscale=0.9]  {$0$};
			\draw (294.66,181.18) node [anchor=north west][inner sep=0.75pt]  [font=\scriptsize,rotate=-90.3,xscale=0.9,yscale=0.9]  {$\cdots $};
			\draw (365.2,147.26) node [anchor=north west][inner sep=0.75pt]  [font=\scriptsize,xscale=0.9,yscale=0.9]  {$0$};
			\draw (365.09,165.48) node [anchor=north west][inner sep=0.75pt]  [font=\scriptsize,xscale=0.9,yscale=0.9]  {$0$};
			\draw (364.65,206.15) node [anchor=north west][inner sep=0.75pt]  [font=\scriptsize,xscale=0.9,yscale=0.9]  {$0$};
			\draw (374.16,181.18) node [anchor=north west][inner sep=0.75pt]  [font=\scriptsize,rotate=-90.3,xscale=0.9,yscale=0.9]  {$\cdots $};

		\end{tikzpicture}

	\end{center}
	
	\noindent where the internal 0-marked squares of the left and right outer squares match and can be composed horizontally. In that case we can use the exchange identity to re-arrange the above horizontal composition into a vertical subdivision of 0-marked squares. Example \cite[Example 4.1]{yo2} shows that this is not always the case and that there might exist horizontally composable squares $\varphi,\psi$ such that any two vertical subdivisions into 0-squares look like:
	
	\begin{center}

		\tikzset{every picture/.style={line width=0.75pt}} 
		
		\begin{tikzpicture}[x=0.75pt,y=0.75pt,yscale=-1,xscale=1]
			
			\draw   (271.59,161) -- (350.41,161) -- (350.41,239.82) -- (271.59,239.82) -- cycle ;
			\draw    (271.66,180.43) -- (350.55,180.14) ;
			\draw    (271.89,197.58) -- (350.11,197.53) ;
			\draw    (271.44,220.21) -- (350.55,220.14) ;
			\draw   (351.09,161) -- (429.91,161) -- (429.91,239.82) -- (351.09,239.82) -- cycle ;
			\draw    (351.16,190.15) -- (430.27,190.7) ;
			\draw    (349.8,211.35) -- (429.85,211.23) ;
			
			\draw (305.7,167.26) node [anchor=north west][inner sep=0.75pt]  [font=\scriptsize,xscale=0.9,yscale=0.9]  {$0$};
			\draw (305.59,183.48) node [anchor=north west][inner sep=0.75pt]  [font=\scriptsize,xscale=0.9,yscale=0.9]  {$0$};
			\draw (305.15,226.15) node [anchor=north west][inner sep=0.75pt]  [font=\scriptsize,xscale=0.9,yscale=0.9]  {$0$};
			\draw (314.66,201.18) node [anchor=north west][inner sep=0.75pt]  [font=\scriptsize,rotate=-90.3,xscale=0.9,yscale=0.9]  {$\cdots $};
			\draw (384.63,171.83) node [anchor=north west][inner sep=0.75pt]  [font=\scriptsize,xscale=0.9,yscale=0.9]  {$0$};
			\draw (385.22,222.15) node [anchor=north west][inner sep=0.75pt]  [font=\scriptsize,xscale=0.9,yscale=0.9]  {$0$};
			\draw (394.44,192.61) node [anchor=north west][inner sep=0.75pt]  [font=\scriptsize,rotate=-90.3,xscale=0.9,yscale=0.9]  {$\cdots $};

		\end{tikzpicture}

	\end{center}
	
	\noindent i.e. the internal 0-squares cannot be arranged to match horizontally. Such horizontal compositions are not 1-marked. We represent squares in $H^1_C$ as above, i.e. squares in $H^1_C\setminus V^1_C$ as squares marked with 1+1/2, i.e. as:

	\begin{center}

		\tikzset{every picture/.style={line width=0.75pt}} 
		
		\begin{tikzpicture}[x=0.75pt,y=0.75pt,yscale=-1,xscale=1]
			
			\draw   (292.77,179.6) -- (369.82,179.6) -- (369.82,256.65) -- (292.77,256.65) -- cycle ;
			
			\draw (310.95,214.72) node [anchor=north west][inner sep=0.75pt]  [font=\scriptsize,xscale=0.9,yscale=0.9]  {$1+1/2$};

		\end{tikzpicture}

	\end{center}
	
	\noindent $V^2_C$ is thus the category of squares admitting a vertical subdivision into squares marked with $\leq 1+1/2$. Inductively, given $k\geq 1$, $V^k_C$ is the category of squares admitting vertical subdivisions as:
	
	\begin{center}

		\tikzset{every picture/.style={line width=0.75pt}} 
		
		\begin{tikzpicture}[x=0.75pt,y=0.75pt,yscale=-1,xscale=1]
			
			\draw   (292.34,180.25) -- (371.16,180.25) -- (371.16,259.07) -- (292.34,259.07) -- cycle ;
			\draw    (292.41,199.68) -- (371.3,199.39) ;
			\draw    (292.64,219.68) -- (370.86,219.63) ;
			\draw    (292.19,239.46) -- (371.3,239.39) ;
			
			\draw (326.45,186.51) node [anchor=north west][inner sep=0.75pt]  [font=\scriptsize,xscale=0.9,yscale=0.9]  {$i_{1}$};
			\draw (326.34,204.73) node [anchor=north west][inner sep=0.75pt]  [font=\scriptsize,xscale=0.9,yscale=0.9]  {$i_{2}$};
			\draw (325.9,245.4) node [anchor=north west][inner sep=0.75pt]  [font=\scriptsize,xscale=0.9,yscale=0.9]  {$i_{s}$};
			\draw (335.41,220.43) node [anchor=north west][inner sep=0.75pt]  [font=\scriptsize,rotate=-90.3,xscale=0.9,yscale=0.9]  {$\cdots $};

		\end{tikzpicture}

	\end{center}

	\noindent where the $i_j$'s are all $\leq k-1/2$. Squares marked with $k$ are squares in $V^k_C$ not marked with $i<k$. $H^{k+1}_C$ is the (possibly weak) category of squares admitting a horizontal subdivision as:

	\begin{center}

		\tikzset{every picture/.style={line width=0.75pt}} 
		
		\begin{tikzpicture}[x=0.75pt,y=0.75pt,yscale=-1,xscale=1]
			
			\draw   (369.66,181.75) -- (369.66,260.57) -- (290.84,260.57) -- (290.84,181.75) -- cycle ;
			\draw    (350.22,181.82) -- (350.52,260.71) ;
			\draw    (330.22,182.05) -- (330.27,260.27) ;
			\draw    (310.45,181.6) -- (310.52,260.71) ;
			
			\draw (295.45,214.51) node [anchor=north west][inner sep=0.75pt]  [font=\scriptsize,xscale=0.9,yscale=0.9]  {$i_{1}$};
			\draw (314.34,214.23) node [anchor=north west][inner sep=0.75pt]  [font=\scriptsize,xscale=0.9,yscale=0.9]  {$i_{2}$};
			\draw (355.4,216.4) node [anchor=north west][inner sep=0.75pt]  [font=\scriptsize,xscale=0.9,yscale=0.9]  {$i_{s}$};
			\draw (349.23,226.06) node [anchor=north west][inner sep=0.75pt]  [font=\scriptsize,rotate=-180.3,xscale=0.9,yscale=0.9]  {$\cdots $};

		\end{tikzpicture}

	\end{center}
	
	\noindent where the $i_j$'s are all $\leq k$. Squares marked with $k+1/2$ are those squares in $H^k_C$ such that no subdivision as above can be reduced as a vertical subdivision as $i$-squares with $i\leq k-1/2$. In \cite{yo2} it shown that there exist globularly generated double categories such that squares marked with $k+1/2$ exist for every $k\geq 0$. The formula $C_1=\varinjlim V^k_C$ thus means that in a globularly generated double category $C$ every square admits a $\mathbb{N}+1/2\mathbb{N}$-marking as above. The length of a square $\varphi$ marked by $x\in \mathbb{N}+1/2\mathbb{N}$ is $\lceil x\rceil$ and the length $\ell C$ is the maximum of legths of squares in $C$. The above pictorial representation is only meant to serve as intuition for the vertical filtration construction and we will not use it for the remainder of the paper.
	
	\
	
	\noindent\textit{Free globularly generated double categories}

	\

	\noindent The free globularly generated double category construction associates to every decorated bicategory $(\mathcal{B}^*,\mathcal{B})$ a globularly generated double category $Q_{(\mathcal{B}^*,\mathcal{B})}$. The double category $Q_{(\mathcal{B}^*,\mathcal{B})}$ lifts the bicategory structure of $\mathcal{B}$ in the sense that the category of objects $Q_{(\mathcal{B}^*,\mathcal{B})_0}$ of $Q_{(\mathcal{B}^*,\mathcal{B})}$ is equal to $\mathcal{B}^*$, the horizontal morphisms of $Q_{(\mathcal{B}^*,\mathcal{B})}$ are the 1-cells of $\mathcal{B}$ and $\mathcal{B}$ is a sub-bicategory of $HQ_{(\mathcal{B}^*,\mathcal{B})_0}$. The equation 
	
	\[H^*Q_{(\mathcal{B}^*,\mathcal{B})}=(\mathcal{B}^*,\mathcal{B})\]

	\noindent holds only in special cases, e.g. $\mathcal{B}^*$ is reduced or $\mathcal{B}^*$ is the category of factors and unital $\ast$-morphisms, but the inclusion
	
	\[(\mathcal{B}^*,\mathcal{B})\subseteq H^*Q_{(\mathcal{B}^*,\mathcal{B})}\]
	
	\noindent always holds. Free globularly generated double categories thus not always provide solutions to Problem \ref{prob}. An example where the above inclusion is proper is provided in \cite[Example 3.1]{yo2}, where it is proven that in the case in which $\mathcal{B}^*$ is the delooping groupoid $\Omega \mathbb{Z}_2$ of $\mathbb{Z}_2$ and $\mathcal{B}$ is the double delooping 2-group $2\Omega \mathbb{Z}_2$ of $\mathbb{Z}_2$, i.e. when $\Omega\mathbb{Z}_2$ is the groupoid with a single object having $\mathbb{Z}_2$ as group of automorphisms and $2\Omega\mathbb{Z}_2$ is the 2-group having a single object with endomorphism category $\Omega\mathbb{Z}_2$, the horizontal bicategory $HQ_{(2\Omega\mathbb{Z}_2,\Omega\mathbb{Z}_2)}$ associated to the decorated bicategory $(2\Omega\mathbb{Z}_2,\Omega\mathbb{Z}_2)$ is equal to $\Omega(\mathbb{Z}_2\ast\mathbb{Z}_2)$. The inclusion $\mathcal{(B,\mathcal{B}^*)}\subseteq H^*Q_\mathcal{(B,\mathcal{B}^*)}$ is in this case obviously proper. We call decorated bicategories for which their free globularly generated double category provides solutions to Problem \ref{prob} saturated. Every decorated bicategory $(\mathcal{B}^*,\mathcal{B})$ has a saturated decorated bicategory associated to it with the same free globularly generated double category as $(\mathcal{B}^*,\mathcal{B})$. Free globularly generated double categories are related to free products and free double categories in the sense of \cite{DawsonPareFree}. Moreover, free globularly generated double categories provide examples of double categories of arbitrarily large and infinite length and provide formal equivariant functorial extensions of the Haagerup standard form and the Connes fusion operation in the theory of representation of von Neumann algebras.

	\

	\noindent \textit{The canonical double projection}
	
	\

	\noindent The canonical double projection construction relates free globularly generated double categories to general solutions to Problem \ref{prob}. Precisely, given a decorated bicategory $(\mathcal{B}^*,\mathcal{B})$ and a double category $C$ satisfying the equation $H^*C=(\mathcal{B}^*,\mathcal{B})$ the double canonical projection associated to $C$ is a strict double functor $\pi^C:Q_{(\mathcal{B}^*,\mathcal{B})}\to \gamma C$ satisfying the equation:

	\[\pi^C\restriction_{(\mathcal{B}^*,\mathcal{B})}=id_{(\mathcal{B}^*,\mathcal{B})}\]

	\noindent and such that $\pi^C$ is surjective on squares. Moreover, $\pi^C$ is unique with respect to this property. We interpret the existence of such double functors as the fact that every globularly generated solution to Problem \ref{prob} for a decorated bicategory $(\mathcal{B}^*,\mathcal{B})$ can be canonically expressed as a double quotient of $Q_{(\mathcal{B}^*,\mathcal{B})}$. We apply the canonical double projection to length, double groupoids, double deloopings of groups decorated by groups, and to double categories of von Neumann algebras. All applications of the canonical double projections follow the slogan: \textit{Saying something about the free globularly generated double category associated to a decorated bicategory translates to saying something about all its globularly generated internalizations}, the intuition of which clearly follows from the properties defining the canonical double projection.

	The canonical double projection construction provides free globularly generated double categories with the structure of universal bases with respect to the fibration $\gamma$ as follows: We extend the free globularly generated double category construction to a functor $Q:\mbox{\textbf{bCat}}^*\to\mbox{\textbf{gCat}}$ using methods analogous to those used in the construction of the canonical double projection. We prove that the set of canonical double projections $\pi^\bullet=\left\{\pi^C:C\in\mbox{\textbf{gCat}}\right\}$ provides a counit to a left adjunction pair $(Q,H^*)$. We thus obtain a diagram as:
	
	
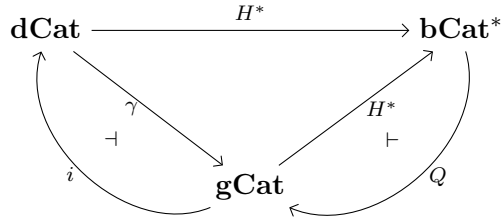
\begin{center}
		
		\begin{tikzpicture}
			\matrix(m)[matrix of math nodes, row sep=4em, column sep=4em,text height=1.5ex, text depth=0.25ex]
			{\mbox{\textbf{dCat}}&&\mbox{\textbf{bCat}}^*\\
				&\mbox{\textbf{gCat}}&\\};
			\path[->,font=\scriptsize,>=angle 90]
			(m-1-1) edge node [above]{$H^*$} (m-1-3)
			edge node [left]{$\gamma$} (m-2-2)
			(m-2-2) edge node [right]{$H^*$}(m-1-3)
			
			(m-2-2) edge [bend left=65] node [black,left]{$i$}(m-1-1)
			(m-2-2) edge [white,bend left=30] node [black, fill=white]{$\dashv$}(m-1-1)

			(m-1-3) edge [bend left=65] node [black,right]{$Q$}(m-2-2)
			(m-1-3) edge [white, bend left=30] node [black, fill=white]{$\vdash$}(m-2-2)

			;
		\end{tikzpicture}
		
	\end{center}

	\noindent completing the similar diagram above. Further, we prove that the restriction $H^*\restriction_{\mbox{\textbf{gCat}}}$ is faithful. This provides \textbf{gCat} with the structure of a concrete category over \textbf{bCat}$^*$ and provides $Q$ with the structure of a free contruction with respect to $H^*$. 
	
	We consider the above statement as a generalization of a classic result in nonabelian algebraic topology. In \cite{BrownSpencer} the concept of edge symmetric double category with connection is introduced. In \cite{Spencer} and later in \cite{BrownMosa} it is proven that the category \textbf{dCat}$^!$ of edge symmetric double categories with connection is equivalent to the category \textbf{2Cat} of 2-categories, with equivalences provided by the horizontalization functor $H$ and the functor associating to every 2-category $B$ its Ehresmann category of quintets $\mbox{\textbf{Q}}B$. Pictorially $H$ and \textbf{Q} fit into a diagram of the form:

	\begin{center}
		
		\begin{tikzpicture}
			\matrix(m)[matrix of math nodes, row sep=4em, column sep=6em,text height=1.5ex, text depth=0.25ex]
			{\mbox{\textbf{dCat}}^!&\mbox{\textbf{2Cat}}\\};
			\path[->,font=\scriptsize,>=angle 90]
			(m-1-1) edge [ bend right=55]node [black,below]{$H$} (m-1-2)
			
			(m-1-2) edge [bend right=55] node [black,above]{\textbf{Q}} (m-1-1)
			(m-1-1) edge [white]node[black,fill=white]{$\cong$} (m-1-2)

			;
		\end{tikzpicture}
	\end{center}

	\noindent The above diagram can be considered as a statement on fillings of hollow squares. When considering problems of filling squares through data provided by general decorated bicategories and not just by data provided by 2-categories decorated by 1-cells, one wishes to obtain a similar statement. We regard the diagram involving $H^*$ and $Q$ above as a decorated bicategory version of the diagram involving \textbf{Q} and $H$ above, fibered by $\gamma$.

	\

	\noindent \textit{Notational conventions}

	\

	\noindent We will follow the notational conventions appearing in \cite{yo1,yo2}. We refer the reader to Section 3 of \cite{yo2} for the details of the notational conventions used in the construction of the free globularly generated double category. We will heavily use the notation and results presented there. In the introduction we have written decorated bicategories in the form $(\mathcal{B}^*,\mathcal{B})$ with $\mathcal{B}^*$ denoting the decoration and $\mathcal{B}$ denoting the underlying bicategory of $(\mathcal{B}^*,\mathcal{B})$ respectively. In what follows we will suppress $\mathcal{B}^*$ from this notation and we will denote $\mathcal{B}$ for a decorated bicategory $(\mathcal{B}^*,\mathcal{B})$.

	\

	\noindent \textit{Contents}

	\

	\noindent In Section \ref{s3} we introduce the canonical double projection construction. We prove that the canonical double projection always exists and that it is uniquely determined by the conditions mentioned in the introduction. The construction of the canonical double projection follows a strategy similar to that of the free globularly generated double category construction. In Section \ref{sApps} we study applications of canonical double projections. We provide upper bounds for lengths of internalizations, we prove that every globularly generated internalization of a decorated 2-groupoid is a double groupoid and we provide 
	compatible formal linear extensions of the Bartels-Douglas-H\'enriques Haagerup standard form and Connes fusion functors to the category of factors and possibly-infinite index morphisms. In Section \ref{sFreeDoubleFunctors} we extend the free globularly generated double category construction to decorated pseudofunctors thus extending the free globularly generated double category construction to a functor. In Section \ref{sFGGCatasafreeobject} we prove that the pair formed by the free globularly generated double category functor and the decorated horizontalization functor forms a left adjoint pair. Moreover, we prove that the restriction of the decorated horizontalization functor to globularly generated double categories is faithful. We use this to interpret globularly generated double categories as a concrete category over decorated bicategories and the free globularly generated double category construction as a free object.

	\section{The canonical double projection}\label{s3}

	\noindent In this section we present the canonical double projection construction. Given a decorated bicategory $\mathcal{B}$ the canonical double projection construction associates to every double category $C$ satisfying the equation $H^*C=\mathcal{B}$ a unique strict double functor $\pi^C:Q_\mathcal{B}\to \gamma C$ such that $\pi^C$ acts as the identity on $\mathcal{B}$ and such that $\pi^C$ is surjective on squares. The following is the main theorem of this section.

	\begin{theorem}\label{thmprojection1}
		Let $\mathcal{B}$ be a decorated bicategory. Let $C$ be a globularly generated double category such that $H^*C=\mathcal{B}$. In that case there exists a unique strict double functor $\pi^C:Q_\mathcal{B}\to C$ such that the equation
		
		\[H^*\pi^C\restriction_\mathcal{B}=id_\mathcal{B}\]

		\noindent holds, and such that $\pi^C$ is surjective on squares.
	\end{theorem}


	\noindent Given a double category $C$ satisfying the conditions above for a decorated bicategory $\mathcal{B}$ we will call the double functor $\pi^C$ provided in Theorem \ref{thmprojection1} the canonical double projection associated to $C$. We divide the construction of $\pi^C$ in several steps. We begin by summarizing the free globularly generated double category construction. We do this in order to set notational conventions used throughout the section and the rest of the paper. The exact details of this construction and the corresponding notational conventions can be found in \cite[Section 2]{yo2}.  
	
	\

	\noindent \textit{The free globularly generated double category: Quick summary}

	\

	\noindent Given functions $s,t:X\to Y$ between sets $X$ and $Y$, which we interpret as source and target functions for elements of $X$, we write $X_{s,t}$ for the set of evaluations of finite compatible words of elements of $X$ with respect to different parentheses patterns. Geometrically $X_{s,t}$ is the set of compatible evaluations, with elements of $X$, of the vertices of all Stasheff associahedra \cite{Stasheff1,Stasheff2}. The functions $s,t$ extend to functions $\tilde{s},\tilde{t}:X_{s,t}\to Y$ and concatenation provides a composition $\ast_{s,t}:X_{s,t}\times_Y X_{s,t}\to X_{s,t}$. Given another pair of functions $s',t':X'\to Y'$ as above and a pair of functions $\psi:X\to X',\varphi:Y\to Y'$ intertwining $s,s'$ and $t,t'$, evaluation on $\psi$ provides a function $\mu_{\psi,\varphi}:X_{s,t}\to X'_{s'.t'}$ intertwining $\tilde{s},\tilde{s}'$, $\tilde{t},\tilde{t}'$ and $\ast_{s,t}$, $\ast_{s',t'}$. We apply these conventions to the situation we are interested in as follows.
	
	Let $\mathcal{B}$ be a decorated bicategory. We formally associate to every 2-cell $\varphi$ in $\mathcal{B}$ a diagram of the form:

	
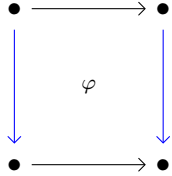
\begin{center}
		
		\begin{tikzpicture}
			\matrix(m)[matrix of math nodes, row sep=4em, column sep=4em,text height=1.5ex, text depth=0.25ex]
			{\bullet&\bullet\\
				\bullet&\bullet\\};
			\path[->,font=\scriptsize,>=angle 90]
			(m-1-1) edge []node {} (m-1-2)
			edge [blue]node{} (m-2-1)
			(m-2-1) edge []node{} (m-2-2)			
			(m-1-2) edge [blue]node{} (m-2-2)
			(m-1-1) edge [white] node [black][fill=white]{$\varphi$}(m-2-2) 
			;
		\end{tikzpicture}
	\end{center}

	\noindent and we associate to every morphism $f$ in $\mathcal{B}^*$ a  square of the form:

	
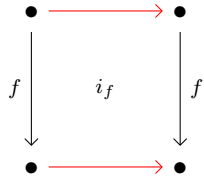
\begin{center}
		
		\begin{tikzpicture}
			\matrix(m)[matrix of math nodes, row sep=4em, column sep=4em,text height=1.5ex, text depth=0.25ex]
			{\bullet&\bullet\\
				\bullet&\bullet\\};
			\path[->,font=\scriptsize,>=angle 90]
			(m-1-1) edge [red]node {} (m-1-2)
			edge node[left]{$f$} (m-2-1)
			(m-2-1) edge [red]node{} (m-2-2)			
			(m-1-2) edge node[right]{$f$} (m-2-2)
			(m-1-1) edge [white] node [black][fill=white]{$i_f$}(m-2-2) 
			;
		\end{tikzpicture}
	\end{center}

	\noindent where the blue and red arrows above always denote identity arrows in $\mathcal{B}^*$ and $\mathcal{B}$ respectively. We write $\mathbb{G}$ for the collection of the above diagrams. The free globularly generated double category $Q_\mathcal{B}$ is the double category freely generated by $\mathbb{G}$. We explain this in more detail. Going around the edges of the above squares there are obvious vertical domain and codomain functions $d_0,c_0:\mathbb{G}\to \mbox{Hom}_{\mathcal{B}_1}$ and obvious horizontal domain and codomain functions $s_0,t_0:\mathbb{G}\to \mbox{Hom}_{\mathcal{B}^*}$. We write $E_1$ for $\mathcal{B}_{1_{s_0,t_0}}$. The functions $d_0,c_0$ extend to functions on $E_1$. We write $F_1$ for the free category generated by $E_1$ with respect to these extensions. The functions $s_0,t_0$ extend to functors on $F_1$. We extend this construction inductively and obtain increasing sequences $E_k$ and $F_k$ equipped with corresponding functions $d_k,c_k$ and functors $s_{k+1},t_{k+1}$ satisfying certain compatibility conditions, see \cite[Lemma 2.5]{yo2}. We consider limits in \textbf{Set} and \textbf{Cat} and obtain a category $F_\infty$ together with functions $d_\infty,c_\infty$ and functors $s_\infty,t_\infty$ extending $d_0,c_0$ and $s_0,t_0$ respectively. The category $F_\infty$ does not capture the information contained in $\mathcal{B}^*$. We thus consider an equivalence relation $R_\infty$ on the set of morphisms $E_\infty$ of $F_\infty$ implementing this information, see \cite[Definition 2.10]{yo2}. We write $V_\infty$ for the quotient $F_\infty/R_\infty$. The structure used to define $F_\infty$ descends to $V_\infty$ and provides the pair $(\mathcal{B}^*,V_\infty)$ with the structure of a double category. This is the free globularly generated double category $Q_\mathcal{B}$ associated to $\mathcal{B}$.

	The category $V_\infty$ described above comes equipped with a filtration $V_k$, which we call the free vertical filtration of $Q_\mathcal{B}$ and the set of squares $H_\infty$ of $Q_\mathcal{B}$ comes equipped with a horizontal filtration $H_k$. We call $H_k$ the free horizontal filtration of $Q_\mathcal{B}$, see \cite[Lemma 2.20]{yo2}. In Section \ref{s4} we deal with the free globularly generated double category associated to more than one decorated bicategory. In that case we will write the corresponding decorated bicategory as superscript in the pieces of structure described above. We now proceed to the proof of Theorem \ref{thmprojection1}. We first briefly explain our strategy for the proof.

	\

	\noindent \textit{Strategy}

	\

	\noindent The construction in Theorem \ref{thmprojection1} will follow a strategy similar to that employed in the free globularly generated double category construction explained above. Let $\mathcal{B}$ be a decorated bicategory. Let $C$ be a globularly generated double category satisfying the equation $H^*C=\mathcal{B}$. We will begin the construction of $\pi^C$ by first defining $\pi^C$ on squares of the form
	
	
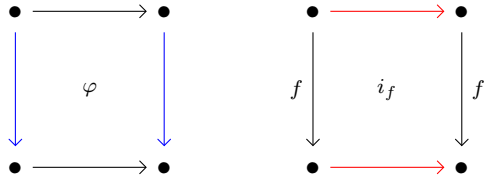
\begin{center}
		
		\begin{tikzpicture}
			\matrix(m)[matrix of math nodes, row sep=4em, column sep=4em,text height=1.5ex, text depth=0.25ex]
			{\bullet&\bullet&\bullet&\bullet\\
				\bullet&\bullet&\bullet&\bullet\\};
			\path[->,font=\scriptsize,>=angle 90]
			(m-1-1) edge []node {} (m-1-2)
			edge [blue]node{} (m-2-1)
			(m-2-1) edge []node{} (m-2-2)			
			(m-1-2) edge [blue]node{} (m-2-2)
			(m-1-1) edge [white] node [black][fill=white]{$\varphi$}(m-2-2)

			(m-1-3) edge [red]node {} (m-1-4)
			edge node[left]{$f$} (m-2-3)
			(m-2-3) edge [red]node{} (m-2-4)			
			(m-1-4) edge node[right]{$f$} (m-2-4)
			(m-1-3) edge[white] node [black][fill=white]{$i_f$}(m-2-4);
		\end{tikzpicture}
	\end{center}

	\noindent Recall that we denote the set of the above squares by $\mathbb{G}$. We thus first define $\pi^C$ on $\mathbb{G}$. The equation

	\[H^*\pi^C\restriction_{\mathcal{B}}=id_\mathcal{B}\]

	\noindent together with the requirement that $\pi^C$ is a strict double functor, forces $\pi^C$ to act as the identity in such squares. We extend $\pi^C$ formally to $E_1$ and we extend this freely to $F_1$. We proceed through an induction argument, to extend $\pi^C$ to $E_k,F_k$ for every positive integer $k$. We do this carefully so as to make these extensions compatible with the finite terms of the structure data $d_k,c_k,s_k,t_k,\ast_k$ defined on categories $F_k$. This is the content of Lemma \ref{leminductionpi1}. We take limits and define a functor on $F_\infty$. We prove that this functor is well defined with respect to the equivalence relation $R_\infty$ defining $Q_\mathcal{B}$. This is the content of Lemma \ref{lemmainductionpi2} and Lemma \ref{compatibilitywithRinftypi}. This will prove that our limit functor descends to a functor from $V_\infty$ to $C_1$. This will be the morphism functor of the canonical double projection $\pi^C$. Finally we take advantage of the vertical filtration on $C$ to prove uniqueness and square surjectivity of $\pi^C$.
	
	We show how the construction of $\pi^C$ works in a specific example. Let $\mathcal{B}$ be the decorated 2-group $(2\Omega\mathbb{Z}_2,\Omega\mathbb{Z}_2)$ as in \cite[Example 3.1]{yo2}. Consider squares of the form
	
	\begin{center}
		
		\begin{tikzpicture}
			\matrix(m)[matrix of math nodes, row sep=4em, column sep=4em,text height=1.5ex, text depth=0.25ex]
			{\ast&\ast\\
				\ast&\ast\\};
			\path[->,font=\scriptsize,>=angle 90]
			(m-1-1) edge [red]node {} (m-1-2)
			edge node[left]{$a$} (m-2-1)
			(m-2-1) edge [red]node{} (m-2-2)			
			(m-1-2) edge node[right]{$a$} (m-2-2)
			(m-1-1) edge [white] node [black][fill=white]{$(a,b)$}(m-2-2) 
			;
		\end{tikzpicture}
	\end{center}
	
	\noindent where $a,b\in\mathbb{Z}_2$. The collection of squares as above forms a double groupoid, which we denote by $C$. The vertical composition of two squares $(a,b)$ and $(a',b')$ in $C$ is the square $(aa',bb')$ and the horizontal composition of two horizontally composable squares $(a,b)$ and $(a,b')$ is $(a,bb')$. It is easily seen that $C$ is globularly generated, has vertical length 1, and that the groupoid of squares of $C$ is the delooping groupoid $\Omega V$ of the Klein 4-group $V$. Moreover, if we identify the 2-cells in $\mathcal{B}$ with the squares $(1,b)$ with $b\in\mathbb{Z}_2$, then the equation $H^*C=\mathcal{B}$ holds. We briefly describe the procedure to construct $\pi^C$ in this case.

	The generating set $\mathbb{G}$ for the free globularly generated double category $Q_\mathcal{B}$ is formed by the formal squares

	\begin{center}
		
		\begin{tikzpicture}
			\matrix(m)[matrix of math nodes, row sep=4em, column sep=4em,text height=1.5ex, text depth=0.25ex]
			{\ast&\ast&\ast&\ast&\ast&\ast\\
				\ast&\ast&\ast&\ast&\ast&\ast\\};
			\path[->,font=\scriptsize,>=angle 90]
			(m-1-1) edge [red]node {} (m-1-2)
			edge node[black,left]{$-1$} (m-2-1)
			(m-2-1) edge [red]node{} (m-2-2)			
			(m-1-2) edge node[black,right]{$-1$} (m-2-2)
			(m-1-1) edge [white] node [black][fill=white]{$i_{-1}$}(m-2-2)
			
			(m-1-3) edge [red]node {} (m-1-4)
			edge [blue]node[black,left]{$1$} (m-2-3)
			(m-2-3) edge [red]node{} (m-2-4)			
			(m-1-4) edge [blue]node[black,right]{$1$} (m-2-4)
			(m-1-3) edge [white] node [black][fill=white]{$i_1$}(m-2-4)
			
			(m-1-5) edge [red]node {} (m-1-6)
			edge [blue]node[black,left]{$1$} (m-2-5)
			(m-2-5) edge [red]node{} (m-2-6)			
			(m-1-6) edge [blue]node[black,right]{$1$} (m-2-6)
			(m-1-5) edge [white] node [black][fill=white]{$-1$}(m-2-6)
			;
		\end{tikzpicture}
	\end{center}

	\noindent The first step in the construction of $\pi^C$ associates to the above squares, from left to right, the following squares in $C$:

	\begin{center}
		
		\begin{tikzpicture}
			\matrix(m)[matrix of math nodes, row sep=4em, column sep=4em,text height=1.5ex, text depth=0.25ex]
			{\ast&\ast&\ast&\ast&\ast&\ast\\
				\ast&\ast&\ast&\ast&\ast&\ast\\};
			\path[->,font=\scriptsize,>=angle 90]
			(m-1-1) edge [red]node {} (m-1-2)
			edge node[black,left]{$-1$} (m-2-1)
			(m-2-1) edge [red]node{} (m-2-2)			
			(m-1-2) edge node[black,right]{$-1$} (m-2-2)
			(m-1-1) edge [white] node [black][fill=white]{$(-1,1)$}(m-2-2)
			
			(m-1-3) edge [red]node {} (m-1-4)
			edge [blue]node[black,left]{$1$} (m-2-3)
			(m-2-3) edge [red]node{} (m-2-4)			
			(m-1-4) edge [blue]node[black,right]{$1$} (m-2-4)
			(m-1-3) edge [white] node [black][fill=white]{$(1,1)$}(m-2-4)
			
			(m-1-5) edge [red]node {} (m-1-6)
			edge [blue]node[black,left]{$1$} (m-2-5)
			(m-2-5) edge [red]node{} (m-2-6)			
			(m-1-6) edge [blue]node[black,right]{$1$} (m-2-6)
			(m-1-5) edge [white] node [black][fill=white]{$(1,-1)$}(m-2-6)
			;
		\end{tikzpicture}
	\end{center}

	\noindent The second step of the free globularly generated double category construction for $\mathcal{B}$ considers the free category $F_1$ on $\mathbb{G}$. In this case $F_1$ is the delooping category on the free monoid generated by the three squares forming $\mathbb{G}$ above. The second step of the construction of $\pi^C$ is thus the unique functor from $F_1$ to $C_1$ extending the value of $\pi^C$ on $\mathbb{G}$ described above. We can recover the square

	\begin{center}
		
		\begin{tikzpicture}
			\matrix(m)[matrix of math nodes, row sep=4em, column sep=4em,text height=1.5ex, text depth=0.25ex]
			{\ast&\ast\\
				\ast&\ast\\};
			\path[->,font=\scriptsize,>=angle 90]
			(m-1-1) edge [red]node {} (m-1-2)
			edge node[left]{$-1$} (m-2-1)
			(m-2-1) edge [red]node{} (m-2-2)			
			(m-1-2) edge node[right]{$-1$} (m-2-2)
			(m-1-1) edge [white] node [black][fill=white]{$(-1,-1)$}(m-2-2) 
			;
		\end{tikzpicture}
	\end{center}

	\noindent in $C$ as the image, under $\pi^C$, of the formal composition

	\begin{center}
		
		\begin{tikzpicture}
			\matrix(m)[matrix of math nodes, row sep=4em, column sep=4em,text height=1.5ex, text depth=0.25ex]
			{\ast&\ast\\
				\ast&\ast\\
				\ast&\ast\\};
			\path[->,font=\scriptsize,>=angle 90]
			(m-1-1) edge [red]node {} (m-1-2)
			edge node[left]{$-1$} (m-2-1)
			(m-2-1) edge [red]node{} (m-2-2)			
			(m-1-2) edge node[right]{$-1$} (m-2-2)
			(m-1-1) edge [white] node [black][fill=white]{$i_{-1}$}(m-2-2)

			(m-2-1) edge [blue]node[black,left]{$1$} (m-3-1)
			(m-3-1) edge [red]node{} (m-3-2)			
			(m-2-2) edge [blue]node[black,right]{$1$} (m-3-2)
			(m-2-1) edge [white] node [black][fill=white]{$-1$}(m-3-2) 
			;
		\end{tikzpicture}
	\end{center}

	\noindent in $F_1$. By \cite[Proposition 5.1]{yo2} the decorated 2-group $\mathcal{B}$ has free length 1 and thus every square in $Q_\mathcal{B}$ can be written as a vertical composition of squares in $F_1$. It is not difficult to see that, $V_\infty$ which in this case is $F^1/R_\infty$, is equal to the delooping groupoid $\Omega(\mathbb{Z}_2\ast\mathbb{Z}_2)$ on the free product $\mathbb{Z}_2\ast\mathbb{Z}_2$ that the canonical double projection $\pi^C$ is the double functor from $Q_\mathcal{B}$ to $C$ induced by the projection from $\mathbb{Z}_1\ast\mathbb{Z}_2$ to $V$ induced by the square-assignments described above. In the case where a decorated bicategory $\mathcal{B}$ has free length $>1$, e.g. \cite[Example 4.1]{yo2} the construction of the canonical double projection $\pi^C$ follows the above pattern inductively.
	
	\

	\noindent \textit{Construction}
	
	\

	\begin{notation}\label{notproj1}
		Let $C$ be a double category. We denote by $q^C$ the function from $\mbox{Hom}_{{C_1}_{s,t}}$ to $\mbox{Hom}_{C_1}$ associating to every evaluation $\Phi$ of a compatible sequence of squares $\Psi_1,...,\Psi_k$ in $C$, the horizontal composition $\Psi_k\ast\dots\ast\Psi_1$ following the parenthesis pattern defining $\Phi$. 
	\end{notation}

	\begin{lem}\label{leminductionpi1}
		Let $\mathcal{B}$ be a decorated bicategory. Let $C$ be a globularly generated double category satisfying the equation $H^*C=\mathcal{B}$. There exists a pair, formed by a sequence of functions $E^\pi_k:E_k\to\mbox{Hom}_{C_1}$ and a sequence of functors $F^\pi_k:F_k\to C_1$, with $k\geq 1$, such that the following conditions are satisfied:
		
		\begin{enumerate}
			\item The restriction $E^\pi_1\restriction_{\mathbb{G}}$ is equal to $id_{\mathbb{G}}$.
			
			\item For every $m,k\geq 1$ such that $m\leq k$, the restriction of $E^\pi_k$ to the set of morphisms of $F_m$ is equal to the morphism function of $F^\pi_m$, and the restriction to $E_m$ of the morphism function of $F^\pi_k$ is equal to $E^\pi_m$.

			\item The following two triangles commute for every positive integer $k$:
			
			\begin{center}
				
				\begin{tikzpicture}
					\matrix(m)[matrix of math nodes, row sep=4em, column sep=3em,text height=1.5ex, text depth=0.25ex]
					{E_k&&\mbox{Hom}_{C_1}\\
						&\mathcal{B}_1&\\};
					\path[->,font=\scriptsize,>=angle 90]
					(m-1-1) edge node[auto] {$E^\pi_k$} (m-1-3)
					edge node[left] {$d_{k},c_k$} (m-2-2)			
					(m-1-3) edge node[right] {$dom,codom$} (m-2-2);
				\end{tikzpicture}

			\end{center}
			
			\item The following two triangles commute for every $k\geq 1$:

			\begin{center}
				
				\begin{tikzpicture}
					\matrix(m)[matrix of math nodes, row sep=4em, column sep=3em,text height=1.5ex, text depth=0.25ex]
					{E_k&&\mbox{Hom}_{C_1}\\
						&\mbox{Hom}_{\mathcal{B}^*}&\\};
					\path[->,font=\scriptsize,>=angle 90]
					(m-1-1) edge node[auto] {$E^\pi_k$} (m-1-3)
					edge node[left] {$s_{k+1},t_{k+1}$} (m-2-2)			
					(m-1-3) edge node[right] {$s,t$} (m-2-2)
					;
				\end{tikzpicture}

			\end{center}

			\item The following two triangles commute for every $k\geq 1$:
			
			\begin{center}
				
				\begin{tikzpicture}
					\matrix(m)[matrix of math nodes, row sep=4em, column sep=3em,text height=1.5ex, text depth=0.25ex]
					{F_k&&C_1\\
						&\mathcal{B}^*&\\};
					\path[->,font=\scriptsize,>=angle 90]
					(m-1-1) edge node[auto] {$F^\pi_k$} (m-1-3)
					edge node[left] {$s_{k+1},t_{k+1}$} (m-2-2)			
					(m-1-3) edge node[auto] {$s,t$} (m-2-2)
					;
				\end{tikzpicture}

			\end{center}

			\item The following square commutes for every $k\geq 1$:
			
			\begin{center}
				
				\begin{tikzpicture}
					\matrix(m)[matrix of math nodes, row sep=4em, column sep=4em,text height=1.5ex, text depth=0.25ex]
					{E_k\times_{\mbox{Hom}_{\mathcal{B}^*}} E_k&\mbox{Hom}_{C_1}\times_{\mbox{Hom}_{\mathcal{B}^*}}\mbox{Hom}_{C_1}\\
						E_k&\mbox{Hom}_{C_1}\\};
					\path[->,font=\scriptsize,>=angle 90]
					(m-1-1) edge node[auto] {$E^\pi_k\times E^\pi_k$} (m-1-2)
					edge node[left] {$\ast_k$} (m-2-1)
					(m-2-1) edge node[below] {$E^\pi_k$} (m-2-2)						
					(m-1-2) edge node[auto] {$\ast$} (m-2-2);
				\end{tikzpicture}

			\end{center}

		\end{enumerate}

		\noindent Moreover, conditions 1-5 above determine the pair of sequences $E^\pi_k$ and $F^\pi_k$.
	\end{lem}
	
	\begin{proof}
		Let $\mathcal{B}$ be a decorated bicategory. Let $C$ be a globularly generated double category such that $H^*C=\mathcal{B}$. We wish to construct a sequence of functions $E^{\pi}_k$ from $E_k^\mathcal{B}$ to Hom$_{C_1}$ and a sequence of functors $F^{\pi}_k$ from $F_k^\mathcal{B}$ to $C_1$ with $k$ running through the collection of positive integers, in such a way that the pair of sequences $E_k^{\pi}$ and $F_k^{\pi}$ satisfies conditions 1-6 of the lemma.
		
		We proceed inductively on $k$. We begin with the definition of function $E_1^{\pi}$. Observe first that from the fact that $H^*C=\mathcal{B}$ it follows that the collection of morphisms of $\mathcal{B}^*$ is equal to the collection of vertical morphisms of $C$. There is thus an obvious identification between the formal horizontal identities of $Q_\mathcal{B}$ and the collection of horizontal identities of $C$. We use this identification and consider the horizontal identities of both $Q_\mathcal{B}$ and $C$ as being the same. Observe that that the equation $H^*C=\mathcal{B}$ also implies that the globular squares of $C$ are precisely the 2-cells of $\mathcal{B}$. Thus $\mathbb{G}$ is the set of generators, as a globularly generated double category, of $C$. We make $E_1^{\pi}$ to be the composition $q^C\mu_{id_\mathbb{G},id_{\mathcal{B}^*}}$. Thus defined $E_1^{\pi}$ is a function from $E_1$ to Hom$_{C_1}$. Moreover, from the way it was defined it easily follows that $E_1^{\pi}$ satisfies condition 1 and conditions 3-5 in the statement the lemma. We now define the functor $F_1^{\pi}$ as follows: Observe first that from the fact that $H^*C=\mathcal{B}$ it follows that the collection of horizontal morphisms of $C$ is equal to $\mathcal{B}_1$. We make the object function of $F_1^{\pi}$ to be $id_{\mathcal{B}_1}$. From the fact that $E_1^{\pi}$ satisfies condition 3 of the statement of the lemma and from the fact that $E_1$ freely generates $F_1$ with respect to $d_1,c_1$ it follows that there exists a unique extension of $E_1^{\pi}$ to a functor from $F_1$ to $C_1$. We make $F_1^{\pi}$ to be this extension. Thus defined $F_1^{\pi}$ trivially satisfies condition 2 of the statement of the lemma with respect to $E_1^{\pi}$. The fact that the functor $F_1^{\pi}$ satisfies the condition 5 in the statement of the lemma follows from the fact that the function $E_1^{\pi}$ satisfies condition 4 and from the functoriality of $s_1$ and $t_1$.
		
		Let $k>1$. Assume now that for every $m<k$ the function $E_m^{\pi}$ from $E_m$ to Hom$_{C_1}$ and the functor $F_m^{\pi}$ from $F_m$ to $C_1$ have been defined, in such a way that the pair of sequences $E_m^{\pi}$ and $F_m^{\pi}$ with $m$ running through the collection of positive integers strictly less than $k$ satisfies the conditions 1-6 in the statement of the lemma. We now construct a function $E_k^{\pi}$ from $E_k$ to Hom$_{C_1}$ and a functor $F_k^{\pi}$ from $F_k$ to $C_1$ such that the pair $E_k^{\pi},F_k^{\pi}$ satisfies conditions 1-6 in the statement of the lemma with respect to the pair of sequences $E_m^{\pi},F_m^{\pi}$ with $m$ running through the collection of positive integers strictly less than $k$.
		
		We first define the function $E_k^{\pi}$. Observe first that from the assumption that $F_{k-1}^{\pi}$ satisfies condition 5 it follows that the function $\mu_{F_{k-1}^{\pi},id_{\mathcal{B}^*}}$ is well defined. We make $E_k^{\pi}$ to be composition $q^C\mu_{F_{k-1}^{\pi},id_{\mathcal{B}^*}}$. Thus defined $E_k^{\pi}$ is a function from $E_k$ to Hom$_{C_1}$. From the way it was defined it is clear that $E_k^{\pi}$ satisfies conditions 4 and 6 of the lemma. From the induction hypothesis it follows that $E_k^{\pi}$ satisfies conditions 1 and 2. The function $E_k^{\pi}$ satisfies the condition 3 of the lemma by the fact that it satisfies condition 2 and by the functoriality of $F_{k-1}^{\pi}$. We now define the functor $F_k^{\pi}$. By the fact that the function $E_k^{\pi}$ satisfies the condition 3 of the lemma it follows that there is a unique extension of $E_k^{\pi}$ to a functor from $F_k$ to $C_1$. We make $F_k^{\pi}$ to be this functor. Thus defined $F_k^{\pi}$ satisfies the condition 2 of the lemma. This follows from the way $F_k^{\pi}$ was constructed and from the fact that condition 2 is already satisfied by the function $E_k^{\pi}$. From the fact that $E_k^{\pi}$ satisfies condition 4 it follows that the functor $F_k^{\pi}$ satisfies the condition 5 of the lemma. We have thus constructed, recursively, a pair of sequences $E_k^{\pi},F_k^{\pi}$ satisfying the conditions in the statement of the lemma. This concludes the proof.
	\end{proof}

	

	\begin{obs}\label{observationinductionpi}
		Let $\mathcal{B}$ be a decorated bicategory. Condition 2 of Lemma \ref{leminductionpi1} implies that for every pair $m,k\geq 1$ such that $m\leq k$ the following two equations hold:

		\[E_k^{\pi}\restriction_{E_m}=E_m^{\pi} \ \mbox{and} \ F_k^{\pi}\restriction_{F_m}=F_m^{\pi}\]
	\end{obs}

	\begin{notation}
		Let $\mathcal{B}$ be a decorated bicategory. Let $C$ be a globularly generated double category satisfying the equation $H^*C=\mathcal{B}$. In that case we write $E_\infty^{\pi}$ for the limit $\varinjlim E_k^{\pi}$ in \textbf{Set} and we write $F_\infty^{\pi}$ for the limit $\varinjlim F_k^{\pi^C}$ \textbf{Cat}. Thus defined $E_\infty^{\pi}$ is a function from $E_\infty$ to the set of squares of $C$ and $F_\infty^{\pi}$ is a functor from $F_\infty$ to the category of squares of $C$. The function $E_\infty^{\pi}$ is the morphism function of $F_\infty^{\pi}$.
	\end{notation}

	\noindent The following lemma follows directly from Lemma \ref{leminductionpi1} and Observation \ref{observationinductionpi}.


	\begin{lem}\label{lemmainductionpi2}
		Let $\mathcal{B}$ be a decorated bicategory. Let $C$ be a globularly generated double category such that $H^*C=\mathcal{B}$. In that case $E_\infty^{\pi}$ and $F_\infty^{\pi}$ satisfy the following conditions:

		\begin{enumerate}
			
			\item The equations $E_\infty^{\pi}\restriction_{E_k}=E_k^{\pi}$ and $F_k^{\pi}\restriction_{F_k}=F_k^{\pi}$ hold for every $k\geq 1$.
			
			\item The following two triangles commute:
			
			\begin{center}
				
				\begin{tikzpicture}
					\matrix(m)[matrix of math nodes, row sep=4em, column sep=3em,text height=1.5ex, text depth=0.25ex]
					{F_\infty&&C_1\\
						&\mathcal{B}^*&\\};
					\path[->,font=\scriptsize,>=angle 90]
					(m-1-1) edge node[auto] {$F_\infty^{\pi}$} (m-1-3)
					edge node[left] {$s_\infty,t_\infty$} (m-2-2)			
					(m-1-3) edge node[auto] {$s,t$} (m-2-2)
					;
				\end{tikzpicture}

			\end{center}

			\item The following square commutes:
			
			\begin{center}
				
				\begin{tikzpicture}
					\matrix(m)[matrix of math nodes, row sep=4em, column sep=4em,text height=1.5ex, text depth=0.25ex]
					{E_\infty\times_{\mbox{Hom}_{\mathcal{B}^*}} E_\infty& \mbox{Hom}_{C_1}\times_{\mbox{Hom}_{\mathcal{B}^*}}\mbox{Hom}_{C_1}\\
						E_\infty&\mbox{Hom}_{C_1}\\};
					\path[->,font=\scriptsize,>=angle 90]
					(m-1-1) edge node[auto] {$E_\infty^{\pi}\times E_\infty^{\pi}$} (m-1-2)
					edge node[left] {$\ast_\infty$} (m-2-1)
					(m-2-1) edge node[below]{$E_\infty^{\pi}$} (m-2-2)							
					(m-1-2) edge node[auto] {$\ast$} (m-2-2);
				\end{tikzpicture}

			\end{center}

		\end{enumerate}

	\end{lem}

	\begin{lem}\label{compatibilitywithRinftypi}
		Let $\mathcal{B}$ be a decorated bicategory. Let $C$ be a globularly generated double category such that $H^*C=\mathcal{B}$. In that case the functor $F_\infty^{\pi}$ is well defined with respect to the equivalence relation $R_\infty$.
	\end{lem}
	
	\begin{proof}
		Let $\mathcal{B}$ be a decorated bicategory. Let $C$ be a globularly generated double category such that $H^*C=\mathcal{B}$. We wish to prove that $F_\infty^{\pi}$ is well defined with respect to the equivalence relation $R_\infty$. The fact that $F_\infty^{\pi}$ is well defined with respect to relation 1 in the definition of $R_\infty$ follows from the functoriality of $F_\infty^{\pi}$ together with the fact that $F_\infty^{\pi}$ satisfies conditions 5 and 6 of Lemma \ref{leminductionpi1}.
		
		We now prove that $F_\infty^{\pi}$ is well defined with respect to relation 2 in the definition of $R_\infty$. Let first $\Phi$ and $\Psi$ be globular squares of $\mathcal{B}$ such that the pair $\Phi,\Psi$ is compatible with respect to $d_\infty$ and $c_\infty$. In that case the image $F_\infty^{\pi}\Psi\bullet_\infty\Phi$ of the vertical composition $\Psi\bullet_\infty\Phi$ of under $F_\infty^{\pi}$ is equal to the image $F_1^{\pi}\Psi\bullet_\infty\Psi$ of $\Psi\bullet_\infty\Phi$ under $F_1^{\pi}$, which is, by functoriality of $F_1^{\pi}$ equal to the composition $F_1^{\pi}\Psi\bullet F_1^{\pi}\Phi$ in $C$ of $F_1^{\pi}\Phi$ and $F_1^{\pi}\Psi$, which is equal, by the definition of $F_1^{\pi}$ to $\Psi\bullet\Phi$. Now, $F_\infty^{\pi}\Psi\bullet\Phi$ is equal to $E_1^{\pi}\Psi\bullet\Phi$, which is equal, by the way $E_1^{\pi}$ was defined, to $\Psi\bullet\Phi$. The functor $F_\infty^{\pi}$ is thus well defined with respect to relation 2 in the definition of $R_\infty$ when restricted to the 2-cells of $\mathcal{B}$. Let now $\alpha$ and $\beta$ be morphisms of $\mathcal{B}$ such that the pair $\alpha,\beta$ is composable. In that case $F_\infty^{\pi}i_\beta\bullet_\infty i_\alpha$ is equal to $F_1^{\pi}i_\beta\bullet_\infty i_\alpha$, which is equal to $F_1^{\pi}i_\beta\bullet F_1^{\pi}i_\beta$. This is equal, again by the definition of $F_1^{\pi}$, to $i_\beta\bullet i_\alpha$. Now, $F_\infty^{\pi}i_{\beta\alpha}$ is equal to $E_1^{\pi}i_{\beta\alpha}$ which is, by the way $E_1^{\pi}$ was defined, equal to $i_{\beta\alpha}$, that is, $F_\infty^{\pi}i_{\beta\alpha}$ is equal to $i_\beta\bullet i_\alpha$. We conclude that $F_\infty^{\pi}$ is well defined with respect to relation 2 in the definition of $R_\infty$ when restricted to formal horizontal identities and thus $F_\infty^{\pi}$ is well defined with respect to relation 2 in the definition of $R_\infty$.
		
		We now prove that $F_\infty^{\pi}$ is well defined with respect to relation 3 in the definition of $R_\infty$. Let $\Phi$ and $\Psi$ be globular squares in $\mathcal{B}$ such that the pair $\Phi,\Psi$ is compatible with respect to $s_\infty$ and $t_\infty$. In that case $F_\infty^{\pi}\Psi\ast_\infty\Phi$ is equal to $E_1^{\pi}\Psi\ast_1\Phi$. This is equal, by the fact that $E_1^{\pi}$ satisfies condition 6 of Lemma \ref{leminductionpi1}, to $E_1^{\pi}\Psi\ast E_1^{\pi}\Phi$, which, by the definition of $E_1^{\pi}$ is equal to $\Psi\ast\Phi$. Now, $F_\infty^{\pi}\Psi\ast\Psi$ is equal to $E_1^{\pi}\Phi\ast\Psi$. This is equal, again by the way $E_1^{\pi}$ was defined, to $\Psi\ast\Phi$. We conclude that $F_\infty^{\pi}$ is well defined with respect to relation 3 in the definition of $R_\infty$.
		
		Finally, the fact that $F_\infty^{\pi}$ is well defined with respect to relations 4 and 5 in the definition of relation $R_\infty$ follows from conditions 3 and 5 of Lemma \ref{leminductionpi1} and from the fact that $id_\mathcal{B}$ carries left and right identity transformations to left and right identity transformations and associators to associators. This concludes the proof of the lemma.

	\end{proof}

	\begin{notation}
		Let $\mathcal{B}$ be a decorated bicategory. Let $C$ be a globularly generated double category such that $H^*C=\mathcal{B}$. In that case will write $V_\infty^{\pi}$ for the functor from $V_\infty$ to $C_1$ induced by $F_\infty^{\pi}$ and $R_\infty$. We write $H_\infty^{\pi}$ for the morphism function of $V_\infty^{\pi}$.
	\end{notation}

	\noindent The proof of the following lemma follows directly from Lemma \ref{lemmainductionpi2} by taking limits.

	\begin{lem}\label{lemmapiexistencefinal}
		Let $\mathcal{B}$ be a decorated bicategory. Let $C$ be a globularly generated double category such that $H^*C=\mathcal{B}$. In that case $V_\infty^{\pi}$ satisfies the following conditions:

		\begin{enumerate}
			
			\item The following two triangles commute:
			
			
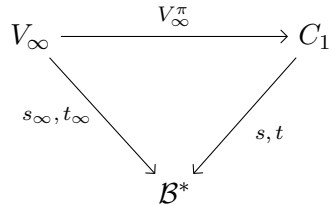
\begin{center}
				
				\begin{tikzpicture}
					\matrix(m)[matrix of math nodes, row sep=4em, column sep=3em,text height=1.5ex, text depth=0.25ex]
					{V_\infty&&C_1\\
						&\mathcal{B}^*&\\};
					\path[->,font=\scriptsize,>=angle 90]
					(m-1-1) edge node[auto] {$V_\infty^{\pi}$} (m-1-3)
					edge node[left] {$s_\infty,t_\infty$} (m-2-2)			
					(m-1-3) edge node[auto] {$s,t$} (m-2-2)
					;
				\end{tikzpicture}

			\end{center}

			\item The following square commutes for every $k\geq 1$:
			
			\begin{center}
				
				\begin{tikzpicture}
					\matrix(m)[matrix of math nodes, row sep=4em, column sep=4em,text height=1.5ex, text depth=0.25ex]
					{V_\infty\times_{\mathcal{B}^*} V_\infty& C_1\times_{\mathcal{B}^*} C_1\\
						V_\infty& C_1\\};
					\path[->,font=\scriptsize,>=angle 90]
					(m-1-1) edge node[auto] {$V_\infty^{\pi}\times V_\infty^{\pi}$} (m-1-2)
					edge node[left] {$\ast_\infty$} (m-2-1)		
					(m-2-1) edge node[below]{$V_\infty^{\pi}$}(m-2-2)					
					(m-1-2) edge node[auto] {$\ast$} (m-2-2);
				\end{tikzpicture}

			\end{center}

		\end{enumerate}
		
	\end{lem}
	
	\

	\noindent \textit{Existence}

	\
	
	\noindent We now prove the existence part of Theorem \ref{thmprojection1}.
	
	\
	
	\noindent \textit{\textbf{Proof:}} Let $\mathcal{B}$ be a decorated bicategory. Let $C$ be a globularly generated double category such that $H^*C=\mathcal{B}$. We wish to construct a double functor $\pi^C:Q_\mathcal{B}\to C$ such that $H^*\pi=d_\mathcal{B}$.
	
	We make $\pi^C$ to be equal to the pair $(id_{\mathcal{B}^*},V_\infty^{\pi})$. The pair $\pi^C$ is a double functor from $Q_\mathcal{B}$ to $C$ by Lemma \ref{lemmapiexistencefinal} and by the fact that it clearly intertwines the horizontal identity functor $i_\infty$ in $Q_\mathcal{B}$ and the horizontal identity functor $i$ in $C$. The fact that $H^*\pi\restriction_{\mathbb{G}}$ is equal to $id_\mathcal{B}$ follows directly from the way $V_\infty^{\pi}$ was defined. This concludes the proof. $\blacksquare$ 
	
	\begin{definition}
		Let $\mathcal{B}$ be a decorated bicategory. Let $C$ be a globularly generated double category such that $H^*C=\mathcal{B}$. We call the double functor $\pi^C$ defined in the above the canonical double projection associated to $C$.
	\end{definition}

	\noindent When necessary we will write $V_\infty^{\pi^C}$ for the morphism functor $V_\infty^\pi$ of the canonical double projection associated to a globularly generated double category $C$. We will use the same convention for $F^\pi_k,H^\pi_k,V^\pi_k$ and $H^\pi_k$
	
	\

	\noindent \textit{Surjectivity}

	\

	\noindent We now prove the surjectivity on squares part of Theorem \ref{thmprojection1}. We begin with the following lemma.

	\begin{lem}\label{lemmasurjectivity}
		Let $\mathcal{B}$ be a decorated bicategory. Let $C$ be a globularly generated double category such that $H^*C=\mathcal{B}$. Let $k$ be a positive integer. The image of $H_\infty^{\pi^C}\restriction_{H_k}$ is equal to $H^k_C$ and the image category of $V_\infty^{\pi}\restriction_{V_k}$ is equal to $V^k_C$.
	\end{lem}
	
	\begin{proof}
		Let $\mathcal{B}$ be a decorated bicategory. Let $C$ be a globularly generated double category such that $H^*C=\mathcal{B}$. Let $k$ be a positive integer. We wish to prove that the image of $H_\infty^{\pi}\restriction_{H_k}$ is equal to $H^k_C$ and that the image category of $V_\infty\restriction_{V_k}$ is equal to $V^k_C$ of vertical filtration associated to $C$. We proceed by induction on $k$.
		
		We prove first that $H_\infty^{\pi}H_1$ is equal to $H^1_C$. From the obvious fact that $H^1_C$ is contained in $H_1$, and from the fact that $\pi$ is a double functor, it follows that $H_\infty H^1_C$ is contained in $H^1_C$. Now, $H_\infty^{\pi}$ acts as the identity function when restricted to 2-cells and horizontal identities of $\mathcal{B}$. It follows, from this, from the fact that $H_\infty^{\pi}$ satisfies condition 2 of lemma \ref{lemmainductionpi2}, and from the way $H^1_C$ is defined, that $H_\infty H^1_C$. We conclude that $\pi H_1$ is equal to $H^1_C$. We now prove that the image category of $V_1$ under $V_\infty^{\pi}$ is equal to $V^1_C$. From the previous argument, and from the fact that $V_\infty^{\pi}$ satisfies condition 2 of lemma \ref{lemmainductionpi2} it follows that $V_\infty^{\pi}H_1$ is equal to $H^1_C$. This, together with the fact that $V_\infty^{\pi}$ is a functor, implies that the image category of $V_1^{\pi}$ under $V_\infty^{\pi}$ is precisely $V^1_C$.

		Let now $k$ be a positive integer such that $k>1$. Suppose that for every $m<k$, $H_\infty^{\pi}H_m$ is equal to $H^m_C$ and that the image category of $V_m$, under $V_\infty^{\pi}$, is equal to $V^m_C$. We now prove that $H_\infty^{\pi}H_k$ is $H^k_C$. From the fact that $H_k$ is obviously contained in $H_k$ and from the fact that $\pi$ is a double functor it follows that $H_\infty^{\pi}H_k$ is contained in $H^k_C$. Now, $H_\infty^{\pi}$ satisfies condition 1 of Lemma \ref{lemmainductionpi2}, the induction hypothesis implies that $H_\infty^{\pi}$Hom$_{V_{k-1}}$ is precisely Hom$_{V^{k-1}_C}$. It follows, from this, from the fact that $H_\infty^{\pi}$ satisfies condition 3 of lemma \ref{lemmainductionpi2} and from the fact that every square in $H^k_C$ is the horizontal composition of a composable sequence of squares in Hom$_{V^{k-1}_C}$ that $H_\infty^{\pi}H_k$ contains $H^k_C$. We thus conclude that $H_\infty^{\pi}H_k$ is equal to $H^k_C$. Finally, we prove that the image category, under $V_\infty^{\pi}$, of $V_k$ is precisely $V^k_C$. From the the previous argument, from Observation \ref{observationinductionpi} and from the fact that $V_\infty^{\pi}$ satisfies condition 1 of Lemma \ref{lemmainductionpi2} it follows that the image of $H_k$ under $V_\infty^{\pi}$ is equal to $H^k_C$. This, together with functoriality of $V_k^{\pi}$ implies that the image category of the restriction to $V_k$, of $V_\infty^{\pi}$, is equal to $V^k_C$. This concludes the proof. 
	\end{proof}

	\noindent We now prove the surjectivity part of Theorem \ref{thmprojection1}.

	\

	\noindent \textit{\textbf{Proof:}} Let $\mathcal{B}$ be a decorated bicategory. Let $C$ be a globularly generated double category such that $H^*C=\mathcal{B}$. We wish to prove that $V_\infty^{\pi}$ is full.
	
	Let $k$ be a positive integer. The restriction, to $V_k$ of $V_\infty$ defines, by Lemma \ref{lemmasurjectivity}, a functor from $V_k$ to $V^k_C$. We denote this functor by $\tilde{V}_k^{\pi}$. The fact that $V_\infty^{\pi}$ satisfies condition 1 of Lemma \ref{lemmainductionpi2} implies that for every pair of integers $m,k$ such that $m\leq k$, the functor $\tilde{V}_m^{\pi}$ is equal to the restriction, to $\tilde{V}_m$, of $\tilde{V}_k^{\pi}$. The sequence $\tilde{V}_k^{\pi}$ is thus a directed system in \textbf{Cat}. The functor $V_\infty^{\pi}$ is equal to its limit $\varinjlim \tilde{V}_k^{\pi}$ in \textbf{Cat}. This, together with the fact, following Lemma \ref{lemmasurjectivity}, that $\tilde{V}_k^{\pi}$ is full for every positive integer $k$ completes the proof of the proposition. $\blacksquare$

	\
	
	\noindent\textit{Uniqueness}
	
	\

	\noindent We begin the proof of uniqueness part of Theorem \ref{thmprojection1} by extending the notation used in the above proof.

	\begin{notation}
		Let $\mathcal{B}$ be a decorated bicategory. Let $C$ be a globularly generated double category. Let $T:Q_\mathcal{B}\to C$ be a double functor. Let $k$ be a positive integer. We write $\tilde{H}_k^T$ for $H_k^T\restriction_{H_k}$. Thus defined $\tilde{H}_k^T$ is a function from $H_k$ to $H^k_C$. Moreover, we write $\tilde{V}_k^T$ for $V_k^T\restriction_{V_k}$. Thus defined $\tilde{V}_k^T$ is a functor from $V_k$ to the $k$-th vertical category $V^k_C$ of $C$.
	\end{notation}

	\begin{lem}\label{lemmauniqueness}
		Let $\mathcal{B}$ be a decorated bicategory. Let $C$ be a globularly generated double category. Let $T,L:Q_\mathcal{B}\to C$ be double functors. If $\tilde{H}_1^T$ and $\tilde{H}_1^L$ are equal, then for every $k\geq 1$, $\tilde{H}_k^T$ and $\tilde{H}_k^L$ are equal and $\tilde{V}_k^T$ and $\tilde{V}_k^L$ are equal. 
	\end{lem}

	\begin{proof}
		Let $\mathcal{B}$ be a decorated bicategory. Let $C$ be a globularly generated double category. Let $T,L:Q_\mathcal{B}\to C$ be double functors. Let $k>1$. Suppose that $\tilde{H}_1^T=\tilde{H}_1^L$. We wish to prove the equations $\tilde{H}_k^T=\tilde{H}_k^L$ and $\tilde{V}_k^T=\tilde{V}_k^L$.
		
		We proceed by induction on $k$. We first prove that $\tilde{V}_1^T=\tilde{V}_1^L$. Observe first that the restriction of the morphism function of $\tilde{V}_1^T$ to $H_1$ is equal to $\tilde{H}_1^T$ and that the restriction of the morphism function of $\tilde{V}_1^L$ to $H_1$ is equal to $\tilde{H}_1^L$. From this and from the assumption of the lemma it follows that the restrictions of the morphism functions of $\tilde{V}_1^T$ and $\tilde{V}_1^L$ to $H_1$ are equal. We conclude, from this, from the fact that $H_1$ generates $V_1$, and from the functoriality of $\tilde{V}_1^T$ and $\tilde{V}_1^L$, that $\tilde{V}_1^T$ and $\tilde{V}_1^L$ are equal. 
		
		Let now $k>1$. Suppose that for every $m<k$ the equations $\tilde{H}_m^T=\tilde{H}_m^L$ and $\tilde{V}_m^T=\tilde{V}_m^L$ hold. We now prove that the equation $\tilde{H}_k^T=\tilde{H}_k^L$ holds. Observe first that the restriction of $\tilde{H}_k^T$ to Hom$_{V_{k-1}^\mathcal{B}}$ is equal to the morphism function of $\tilde{V}_{k-1}^T$ and that the restriction of $\tilde{H}_k^L$ to Hom$_{V_{k-1}^\mathcal{B}}$ is equal to the morphism function of $\tilde{V}_{k-1}^L$. From this, from induction hypothesis, and from the fact that both $T$ and $L$ are double functors that the equation $\tilde{H}_k^T=\tilde{H}_k^L$ holds. We now prove the equation $\tilde{V}_k^T=\tilde{V}_k^L$ holds. Observe again that the restriction of the morphism function of $\tilde{V}_k^T$ to $H_k$ is equal to $\tilde{H}_k^T$ and that the restriction of the morphism function of $\tilde{V}_k^L$ to $H_k$ is equal to $\tilde{H}_k^L$. From this, from the previous argument, from the fact that $H_k^\mathcal{B}$ generates the category $V_k$, and from the functoriality of $\tilde{V}_k^T$ and $\tilde{V}_k^L$ it follows that the equation $\tilde{V}_k^T=\tilde{V}_k^L$ holds. This concludes the proof. 
	\end{proof}

	\noindent Given a double functor $T:Q_\mathcal{B}\to C$ from the free globularly generated double category associated to a decorated bicategory $\mathcal{B}$ to a globularly generated double category $C$, it is a straightforward observation that the morphism functor $T_1$ of $T$ is equal to $\varinjlim \tilde{V}_k^T$ in \textbf{Cat}. This, together with Lemma \ref{lemmauniqueness} implies the following proposition. We interpret this by saying that a double functor with domain a free globularily generated double category is completely determined by its value on globular squares.

	\begin{prop}\label{uniquenessprop}
		Let $\mathcal{B}$ be a decorated bicategory. Let $C$ be a globularly generated double category. Let $T,L:Q_\mathcal{B}\to C$ be double functors. If $\tilde{H}_1^T$ and $\tilde{H}_1^L$ are equal then $T_1$ and $L_1$ are equal.
	\end{prop}

	\noindent The uniqueness part of Theorem \ref{thmprojection1} follows directly from the above proposition. We interpreted the surjectivity part of Theorem \ref{thmprojection1} by saying that every globularly generated internalization of a decorated bicategory $\mathcal{B}$ could be interpreted as a quotient of the free globularly generated double category $Q_\mathcal{B}$ associated to $\mathcal{B}$ via the canonical projection double functor. We interpret the uniqueness part of Theorem \ref{thmprojection1} by saying that in this case the choice of canonical projections as projection is canonical.

	\

	\noindent \textit{Linear canonical double projection}

	\

	\noindent Let $k$ be a field. Let $\mathcal{B}$ be a $k$-linear decorated bicategory. In that case the free globularly generated double category construction can be modified to produce a $k$-linear free globularly generated double category $Q^k_\mathcal{B}$ associated to $\mathcal{B}$, see the final comments of \cite[Section 2]{yo2}. Given $k$-linear decorated bicategories $\mathcal{B},\mathcal{B}'$ we will say that a decorated pseudofunctor $G:\mathcal{B}\to\mathcal{B}'$ is linear if $G$ is linear on 2-cells and vertical arrows of $\mathcal{B}$. It is easily seen that the canonical double projection $\pi^C$ associated to a linear globularly generated double category $C$ satisfying the equation $H^*C=\mathcal{B}$ is a linear pseudofunctor. We will make use of this fact in the next section.

	
	
	


	\section{Applications}\label{sApps}

	\noindent In this section we make use of the canonical double projection to obtain information about solutions to Problem \ref{prob}. We study applications of Theorem \ref{thmprojection1} to length, double groupoids, single 1- and 2-cell decorated bicategories and double categories of von Neumann algebras.
	
	\
	
	\noindent \textit{Length}
	
	\
	
	\noindent Recall that the length of a globularly generated double category $C$, $\ell C$, is the minimal $k\in \mathbb{N}\cup\left\{\infty\right\}$ for which $V^k_C=C_1$. In the non-globularly generated case we define the length of a double category $C$ as $\ell \gamma C$. The length of a double category $C$ is meant to serve as a measure of complexity on the interplay between horizontal and vertical compositions of globular and horizontal squares of $C$. Equivalently $\ell C$ serves as a measure of complexity on presentations of globularly generated squares of $C$. Double categories of arbitrarily large and infinite lengths were constructed in \cite{yo2}. Using the free globularly generated double category construction we translate the definition of length to decorated bicategories. Given a decorated bicategory $\mathcal{B}$ we define the length $\ell \mathcal{B}$ of $\mathcal{B}$ as $\ell Q_\mathcal{B}$. We prove the following proposition.

	\begin{prop}\label{propapplicationlength}
		Let $\mathcal{B}$ be a decorated bicategory. Let $C$ be a double category. If $H^*C=\mathcal{B}$ then the following inequality holds:
		
		\[\ell C\leq \ell\mathcal{B}\]

	\end{prop}

	\begin{proof}
		Let $\mathcal{B}$ be a decorated bicategory. Let $C$ be a double category such that $H^*C=\mathcal{B}$. We wish to prove that $\ell C\leq \ell\mathcal{B}$.

		From the equations $H^*\gamma C=H^*C$ and $\ell\gamma C=\ell C$ we may assume that $C$ is globularly generated. Let $k$ be a positive integer. Suppose $\ell\mathcal{B}=k$. We wish to prove that $\ell C\leq k$. To prove this it is enough to prove that $H^{k+1}_C$ is closed under vertical compositions. Let $\varphi,\psi$ be vertically compatible squares in $H^{k+1}_C$. We wish to prove that $\varphi\bullet\psi\in H^{k+1}_C$. By the fact that $\pi^C$ is surjective on squares the function $H^\pi_{k+1}:H^{k+1}_{Q_\mathcal{B}}\to H^{k+1}_C$ is epic. Let $\varphi',\psi'\in H^{k+1}_{Q_\mathcal{B}}$ such that
		
		\[H^{\pi}_{k+1}\varphi'=\varphi \ \mbox{and} \ H^{\pi}_{k+1}\psi'=\psi\]

		\noindent By the fact that $\pi$ intertwines vertical domain and codomains of $Q_\mathcal{B}$ and $C$ it follows that the squares $\varphi',\psi'$ are vertically compatible. From the fact that $\ell\mathcal{B}=\ell Q_\mathcal{B}= k$ it follows that $\varphi'\bullet_\infty\psi'\in H^{k+1}$ and thus the square:

		\[H^{\pi}_{k+1}(\varphi'\bullet_\infty\psi')=\varphi\bullet\psi\]
		
		\noindent is a square in $H^{k+1}_C$. We conclude that $\ell C\leq k$. The case in which $\ell\mathcal{B}=\infty$ is trivial. This concludes the proof of the proposition.
	\end{proof}

	\noindent An important case of Proposition \ref{propapplicationlength} is when the length of the decorated bicategory $\mathcal{B}$ is assumed to be 1. This is contained in the following immediate corollary.
	
	\begin{cor}\label{corollarylengthapplication}
		Let $\mathcal{B}$ be a decorated bicategory. Suppose $\ell\mathcal{B}=1$. If $C$ is a double category such that $H^*C=\mathcal{B}$ then $\ell C=1$.
	\end{cor}

	\noindent Corollary \ref{corollarylengthapplication} says that if we assume that $\ell\mathcal{B}=1$ then we have a good control on expressions of all squares of any globularly generated double category satisfying $H^*C=\mathcal{B}$. More precisely, every square $\varphi$ of a globularly generated double category $C$ satisfying the equation $H^*C=\mathcal{B}$ admits a decomposition as a vertical composition of squares of the following four forms:

	
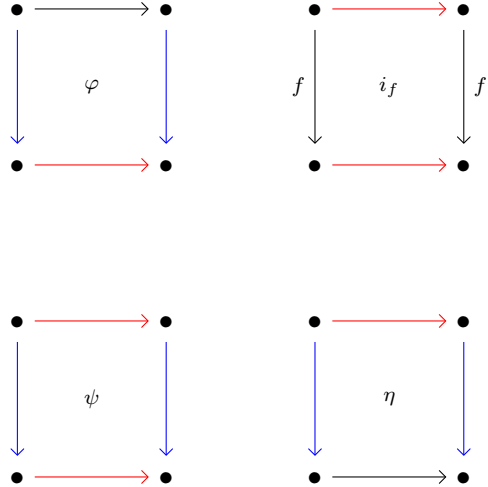
\begin{center}
		
		\begin{tikzpicture}
			\matrix(m)[matrix of math nodes, row sep=4em, column sep=4em,text height=1.5ex, text depth=0.25ex]
			{\bullet&\bullet&\bullet&\bullet\\
				\bullet&\bullet&\bullet&\bullet\\
				\bullet&\bullet&\bullet&\bullet\\
				\bullet&\bullet&\bullet&\bullet\\};
			\path[->,font=\scriptsize,>=angle 90]
			(m-1-1) edge node {} (m-1-2)
			edge [blue]node{} (m-2-1)
			(m-2-1) edge [red]node{} (m-2-2)			
			(m-1-2) edge [blue]node{} (m-2-2)
			(m-1-1) edge [white] node [black][fill=white]{$\varphi$}(m-2-2)

			(m-1-3) edge [red]node {} (m-1-4)
			edge node[left]{$f$} (m-2-3)
			(m-2-3) edge [red]node{} (m-2-4)			
			(m-1-4) edge node[right]{$f$} (m-2-4)
			(m-1-3) edge[white] node [black][fill=white]{$i_f$}(m-2-4)

			(m-3-1) edge [red]node {} (m-3-2)
			edge [blue]node {} (m-4-1)
			(m-4-1) edge [red]node{} (m-4-2)			
			(m-3-2) edge [blue]node {} (m-4-2)
			(m-3-1) edge[white] node [black][fill=white]{$\psi$}(m-4-2)

			(m-3-3) edge [red]node {} (m-3-4)
			edge [blue]node {} (m-4-3)
			(m-4-3) edge node{} (m-4-4)			
			(m-3-4) edge [blue]node {} (m-4-4)
			(m-3-3) edge[white] node [black][fill=white]{$\eta$}(m-4-4);
		\end{tikzpicture}
		
	\end{center}

	\noindent and the horizontal composition of any such squares in $C$ can be re-arranged so as to be written in the above form. The following are examples of decorated bicategories of length 1.

	\

	\begin{enumerate}
		\item \textbf{Groups decorated by groups:} In \cite[Proposition 5.1 ]{yo2} the following equation is proven:

		\[\ell(\Omega G,2\Omega A)=1\]

		\noindent for every pair of groups $G,A$ with $A$ abelian, where recall that $\Omega G$ and $2\Omega G$ are the delooping groupoid and the double delooping 2-group of $G$ respectively, i.e. $\Omega G$ is the groupoid with a single object $\ast$ such that $Aut_{\Omega G}(\ast)=G$ and $2\Omega A$ is the 2-group with a single object, which we also denote by $\ast$, such that $End_{2\Omega A}(\ast)$ is equal to $\Omega A$. Every square $\varphi$ in any globularly generated double category $C$ satisfying the equation $H^*C=(\Omega G,2\Omega A)$ can thus be written as a vertical composition of squares of the form:

		\begin{center}
			
			\begin{tikzpicture}
				\matrix(m)[matrix of math nodes, row sep=4em, column sep=4em,text height=1.5ex, text depth=0.25ex]
				{\ast&\ast&\ast&\ast\\
					\ast&\ast&\ast&\ast\\};
				\path[->,font=\scriptsize,>=angle 90]
				(m-1-1) edge [red]node {} (m-1-2)
				edge [blue]node{} (m-2-1)
				(m-2-1) edge [red]node{} (m-2-2)			
				(m-1-2) edge [blue]node{} (m-2-2)
				(m-1-1) edge [white] node [black][fill=white]{$\xi$}(m-2-2)

				(m-1-3) edge [red]node {} (m-1-4)
				edge node[left]{$g$} (m-2-3)
				(m-2-3) edge [red]node{} (m-2-4)			
				(m-1-4) edge node[right]{$g$} (m-2-4)
				(m-1-3) edge[white] node [black][fill=white]{$i_g$}(m-2-4);
			\end{tikzpicture}
		\end{center}

		\noindent where $\ast$ denotes the only object in $\Omega G$, $\xi$ is an element of the monoid of squares of the only horizontal morphism $i_\ast$ of $C$ and where $g$ is any element of $G$. In Corollary \ref{groupspresentation} we obtain more information about double categories of this form.

		\item \textbf{von Neumann algebras:} In \cite[Proposition 6.1]{yo2} the following equation was proven:

		\[\ell Q^\mathbb{C}_{W^*_{fact}}=1\]

		\noindent where $W^*_{fact}$ denotes the bicategory of factors, Hilbert bimodules and intertwiners, decorated by the category of possibly infinite index unital $^*$-morphisms. Every square in any linear globularly generated double category $C$ satisfying the equation $H^*C=W^*_{fact}$ can thus be written as a multiple of a vertical composition of squares of the form:

		\begin{center}
			
			\begin{tikzpicture}
				\matrix(m)[matrix of math nodes, row sep=4em, column sep=4em,text height=1.5ex, text depth=0.25ex]
				{A&A&A&A&B&B\\
					A&A&B&B&B&B\\};
				\path[->,font=\scriptsize,>=angle 90]
				(m-1-1) edge node [above]{$H$} (m-1-2)
				edge [blue]node{} (m-2-1)
				(m-2-1) edge [red]node{} (m-2-2)			
				(m-1-2) edge [blue]node{} (m-2-2)
				(m-1-1) edge [white] node [black][fill=white]{$\varphi$}(m-2-2)

				(m-1-3) edge [red]node {} (m-1-4)
				edge node[left]{$f$} (m-2-3)
				(m-2-3) edge [red]node{} (m-2-4)			
				(m-1-4) edge node[right]{$f$} (m-2-4)
				(m-1-3) edge[white] node [black][fill=white]{$i_f$}(m-2-4)

				(m-1-5) edge [red]node {} (m-1-6)
				edge [blue]node{} (m-2-5)
				(m-2-5) edge node [below]{$K$} (m-2-6)			
				(m-1-6) edge [blue]node{} (m-2-6)
				(m-1-5) edge[white] node [black][fill=white]{$\psi$}(m-2-6)
				;
			\end{tikzpicture}
		\end{center}

	\end{enumerate}

	\noindent where $A,B$ are factors, $H$ is a left-right $A$-bimodule, $\varphi$ is a bounded intertwiner from $H$ to $L^2(A)$, $f:A\to B$ is a possibly infinite index unital $^*$-morphism, $K$ is a left-right $B$-bimodule, and $\psi$ is a bounded intertwiner from $L^2(B)$ to $K$. In Proposition \ref{canonicalprojectionvnalgebras} we obtain more information of double categories of this form.

	\newpage

	\noindent \textit{2-groupoids and double groupoids}

	\

	\noindent Double groupoids and 2-groupoids categorify crossed modules and are thus used to model homotopy 2-types \cite{BrownSpencer2,MartinezCegarra}. Relations between double groupoids and 2-groupoids have been studied in \cite{BrownSpencer} in the case of edge-symmetric double groupoids with special connection. We apply the results obtained in Section \ref{s3} to study relations between decorated 2-groupoids and general double groupoids. We say that a decorated bicategory $\mathcal{B}$ is a decorated 2-groupoid if $\mathcal{B}$ is a 2-groupoid and $\mathcal{B}^*$ is a groupoid. Decorated bigroupoids are defined analogously. Given a 3-filtered topological space $(X,A,C)$, the pair $(\Pi_1(A;C),W(X;A,C))$, where $W(X;A,C)$ is Moerdijk-Svensson's Whitehead homotopy 2-groupoid associated to $(X,A,C)$ \cite{MoerdijkSvensson} and $\Pi_1(A;C)$ is the fundamental groupoid of $A$ relative to $C$, is a decorated 2-groupoid. The Brown-Higgins fundamental double groupoid $\rho(X;A,C)$ \cite{BrownHiggins} satisfies the equation

	\[H^*\rho(X;A,C)=(\Pi_1(A,C),W(X;A,C))\]

	\noindent Decorated 2-groupoids of the form $(\Pi_1(A;C),W(X;A,C))$ thus always admit solutions to Problem \ref{prob} and these solutions can always be chosen to be double groupoids. A similar statement holds for homotopy 2-groupoids $G_2(X)$ associated to Hausdorff topological spaces $X$ by Hardie, Kamps and Kieboom in \cite{HardieKampsKieboom} decorated by the full fundamental groupoid $\Pi_1(X)$, with internalization provided by the Brown-Hardie-Kamps-Porter homotopy double groupoid $\rho_2^{\square} (X)$ defined in \cite{BrownHardieKampsPorter}.

	Invertibility is perhaps the most essential condition on structures involved in the homotopy hypothesis. In our context it is thus an important question whether every decorated 2-groupoid can always be internalized by a double groupoid. The Brown-Spencer theorem \cite{BrownSpencer} applies in the context of special double groupoids with special connections \cite{BrownSpencer2}, and thus every 2-groupoid $\mathcal{B}$, decorated by its groupoid of horizontal arrows is internalized by a double groupoid, its Ehresmann double category of quintets. We treat the general case of 2-groupoids decorated by groupoids which are not-necessarily groupoids of horizontal arrows. We prove that given a general decorated 2-groupoid (more generally a decorated bigroupoid) $\mathcal{B}$, if there exists a double category $C$ (not-necessarily a double groupoid) such that $H^*C=\mathcal{B}$ then $\gamma C$ is a double groupoid. We begin with the following lemma.

	\begin{prop}\label{propdoublegroupoidfrees}
		Let $\mathcal{B}$ be a decorated bicategory. If $\mathcal{B}$ is a decorated 2-groupoid then $Q_\mathcal{B}$ is a double groupoid.
	\end{prop}

	\begin{proof}
		Let $\mathcal{B}$ be a decorated 2-groupoid. We wish to prove that $Q_\mathcal{B}$ is a double groupoid.
		
		We prove by induction on $k$ that every square $\varphi$ in $V_k$ is vertically and horizontally invertible. By the condition that $\mathcal{B}$ is a decorated 2-groupoid all squares of $Q_\mathcal{B}$ of the form:
		
		\begin{center}
			
			\begin{tikzpicture}
				\matrix(m)[matrix of math nodes, row sep=4em, column sep=4em,text height=1.5ex, text depth=0.25ex]
				{\ast&\ast&\ast&\ast\\
					\ast&\ast&\ast&\ast\\};
				\path[->,font=\scriptsize,>=angle 90]
				(m-1-1) edge [red]node {} (m-1-2)
				edge node[left]{$f$} (m-2-1)
				(m-2-1) edge [red]node{} (m-2-2)			
				(m-1-2) edge node[right]{$f$} (m-2-2)
				(m-1-1) edge [white] node [black][fill=white]{$i_f$}(m-2-2)

				(m-1-3) edge node {} (m-1-4)
				edge [blue]node[left]{} (m-2-3)
				(m-2-3) edge node{} (m-2-4)			
				(m-1-4) edge [blue]node[right]{} (m-2-4)
				(m-1-3) edge[white] node [black][fill=white]{$\varphi$}(m-2-4);
			\end{tikzpicture}
		\end{center}

		\noindent are vertically and horizontally invertible, with the vertical and horizontal inverse of a square on the left-hand side above given by

		\begin{center}
			
			\begin{tikzpicture}
				\matrix(m)[matrix of math nodes, row sep=4em, column sep=4em,text height=1.5ex, text depth=0.25ex]
				{\ast&\ast&\ast&\ast\\
					\ast&\ast&\ast&\ast\\};
				\path[->,font=\scriptsize,>=angle 90]
				(m-1-1) edge [red]node {} (m-1-2)
				edge node[left]{$f^{-1}$} (m-2-1)
				(m-2-1) edge [red]node{} (m-2-2)			
				(m-1-2) edge node[right]{$f^{-1}$} (m-2-2)
				(m-1-1) edge [white] node [black][fill=white]{$i_{f^{-1}}$}(m-2-2)

				(m-1-3) edge [red]node {} (m-1-4)
				edge node[left]{$f$} (m-2-3)
				(m-2-3) edge [red]node{} (m-2-4)			
				(m-1-4) edge node[right]{$f$} (m-2-4)
				(m-1-3) edge[white] node [black][fill=white]{$i_f$}(m-2-4);
			\end{tikzpicture}
		\end{center}

		\noindent respectively. Given any globular or horizontal identity square $\varphi$ in $Q_\mathcal{B}$ we will write $v\varphi^{-1}$ and $h\varphi^{-1}$ for its its vertical and its horizontal inverse in $Q_\mathcal{B}$ respectively. Suppose $\varphi$ is a general square in $V_1$. Write $\varphi$ as a vertical composition of the form $\varphi=\varphi_k\bullet\dots\varphi_1$ where the $\varphi_i's$ are squares of $C$ as above. In that case the vertical inverse of $\varphi$ is given by the composition $v\varphi_{1}^{-1}\bullet\dots\bullet v\varphi_{k}^{-1}$ and the horizontal inverse of $\varphi$ is given by the vertical composition $h\varphi_{k}^{-1}\bullet\dots\bullet h\varphi_{1}^{-1}$. 
		
		Let $n$ be a positive integer such that $n>1$. Suppose that for every $m<n$ every square in $V_m$ is both vertically and horizontally invertible. We prove that every square in $V_n$ is vertically and horizontally invertible. Let $\varphi$ first be a square in $H_n$. Write $\varphi$ as a horizontal composition $\varphi=\varphi_k\ast\dots\ast\varphi_1$ with $\varphi_i$ in $V_{n-1}$. By the induction hypothesis the squares $\varphi_i$ are all vertically and horizontally invertible. We again write $v\varphi^{-1}_i$ and $h\varphi^{-1}_i$ for the horizontal and the vertical inverse of $\varphi_i$ respectively. The vertical inverse $v\varphi^{-1}$ of $\varphi$ is given by the horizontal composition $v\varphi^{-1}_k\ast\dots\ast v\varphi^{-1}_1$ and the horizontal inverse $h\varphi^{-1}$ of $\varphi$ is given by the horizontal composition $h\varphi_1^{-1}\ast\dots\ast h\varphi^{-1}_k$. Thus every square in $H_n$ admits both a horizontal and a vertical inverse. Using this and the same argument used in the previous paragraph every square in $V_n$ is vertically and horizontally invertible. This concludes the proof.
	\end{proof}


	\begin{cor}\label{cordoublegroupoid}
		Let $C$ be a double category. If $H^*C$ is a decorated 2-groupoid then $\gamma C$ is a double groupoid.
	\end{cor}

	\begin{proof}
		Let $C$ be a double category. Suppose that $H^*C$ is a decorated 2-groupoid. We wish to prove that $\gamma C$ is a double groupoid.
		
		It is enough to prove that every square in $\gamma C$ is both vertically and horizontally invertible. This follows directly from Proposition \ref{propdoublegroupoidfrees} and Theorem \ref{thmprojection1}. This concludes the proof of the corollary.
	\end{proof}
	
	\noindent Observe that Proposition \ref{propdoublegroupoidfrees} and Corollary \ref{cordoublegroupoid} still hold if we assume that $\mathcal{B}$ is a decorated bigroupoid. The following corollary follows directly from Proposition \ref{propapplicationlength}, Proposition \ref{propdoublegroupoidfrees}, and \cite[Corollary 5.2]{yo2} by considering decorated 2-groupoids of the form $(\Omega G,2\Omega A)$ whith $G,A$ groups and $A$ abelian.
	
	\begin{cor}\label{groupspresentation}
		Let $G,A$ be groups. Suppose $A$ is abelian. Let $C$ be a globularly generated double category. If $H^*C=(\Omega G,2\Omega A)$ then the category of squares $C_1$ is of the form $\Omega H$ for a group $H$ such that $H$ is a quotient of $G\ast A$. 
	\end{cor}

	\

	\noindent \textit{von Neumann algebras}

	\

	\noindent We study linear double categories of von Neumann algebras and their Hilbert bimodules. In \cite{Bartels1} a tensor double category of semisimple von Neumann algebras, Hilbert bimodules, equivariant intertwiners and finite index morphisms was constructed in order to express the fact that the Haagerup standard form and the Connes fusion operation admit compatible extensions to tensor functors. In \cite{yo2} it was proven that the bicategory of factors, Hilbert bimodules, and intertwiners, decorated by possibly infinite index morphisms is saturated and thus its linear free globularly generated double category is an internalization, providing formal linear functorial extensions of both the Haagerup standard form and the Connes fusion operations. We investigate the relation of these two constructions through the canonical double projection and we use this to construct a linear extension of the double category of factors and finite morphisms accommodating possibly infinite index morphisms.
	
	\
	
	\noindent We write \textbf{Mod}$^{fact}$ for the linear bicategory whose 2-cells are of the form:

	\begin{center}
		
		\begin{tikzpicture}
			\matrix(m)[matrix of math nodes, row sep=4em, column sep=4em,text height=1.5ex, text depth=0.25ex]
			{A&B\\};
			\path[->,font=\scriptsize,>=angle 90]
			(m-1-1) edge [bend right=45]node [below]{$H$} (m-1-2)
			edge [bend left=45] node [above]{$K$} (m-1-2)
			edge [white]node[black][fill=white]{$\varphi$} (m-1-2)

			;
		\end{tikzpicture}
	\end{center}

	\noindent where $A,B$ are factors, $H,K$ are left-right Hilbert $A$-$B$-bimodules and where $\varphi$ is an intertwiner operator from $H$ to $K$. Horizontal identity 2-cells in \textbf{Mod}$^{fact}$ are given by 2-cells of the form:

	\begin{center}
		
		\begin{tikzpicture}
			\matrix(m)[matrix of math nodes, row sep=4em, column sep=4em,text height=1.5ex, text depth=0.25ex]
			{A&A\\};
			\path[->,font=\scriptsize,>=angle 90]
			(m-1-1) edge [red,bend right=45]node[black][below]{$L^2(A)$} (m-1-2)
			edge [red,bend left=45] node[black][above]{$L^2(A)$} (m-1-2)
			edge [white]node[black][fill=white]{$id_{L^2(A)}$} (m-1-2)

			;
		\end{tikzpicture}
	\end{center}

	\noindent where $A$ is a factor and $L^2(A)$ is the Haagerup standard form of $A$, see \cite{Haagerup}. Given horizontally compatible 2-cells in \textbf{Mod}$^{fact}$ of the form:

	\begin{center}
		
		\begin{tikzpicture}
			\matrix(m)[matrix of math nodes, row sep=4em, column sep=4em,text height=1.5ex, text depth=0.25ex]
			{A&B&C\\};
			\path[->,font=\scriptsize,>=angle 90]
			(m-1-1) edge [bend right=45]node [below]{$H$} (m-1-2)
			edge [bend left=45] node [above]{$K$} (m-1-2)
			edge [white]node[black][fill=white]{$\varphi$} (m-1-2)
			(m-1-2) edge [bend right=45]node [below]{$H'$} (m-1-3)
			edge [bend left=45] node [above]{$K'$} (m-1-3)
			edge [white]node[black][fill=white]{$\varphi'$} (m-1-3)
			
			;
		\end{tikzpicture}
	\end{center}

	\noindent their horizontal composition is provided by the 2-cell:
	
	\begin{center}
		
		\begin{tikzpicture}
			\matrix(m)[matrix of math nodes, row sep=4em, column sep=4em,text height=1.5ex, text depth=0.25ex]
			{A&C\\};
			\path[->,font=\scriptsize,>=angle 90]
			(m-1-1) edge [bend right=45]node [below]{$H\boxtimes_BH'$} (m-1-2)
			edge [bend left=45] node [above]{$K\boxtimes_BK'$} (m-1-2)
			edge [white]node[black][fill=white]{$\varphi\boxtimes_B\varphi'$} (m-1-2)

			;
		\end{tikzpicture}
	\end{center}

	\noindent where $H\boxtimes_BH'$ and $K\boxtimes_BK'$ denote the Connes fusion of $H,H'$ and $K,K'$ and where $\varphi\boxtimes_B\varphi'$ denotes the Connes fusion of $\varphi$ and $\varphi'$. We write \textbf{vN}$^{fact}$ for the category of factors and unital $^\ast$-morphisms $f:A\to B$ with $[f(A),B]$ possibly infinite. We write $\mbox{\textbf{vN}}^{fin}$ for the subcategory of \textbf{vN}$^{fact}$ generated by $^*$-morphisms $f:A\to B$ such that $[f(A),B]<\infty$. The pairs $(\mbox{\textbf{vN}}^{fact},\mbox{\textbf{Mod}}^{fact})$ and $(\mbox{\textbf{vN}}^{fin},\mbox{\textbf{Mod}}^{fin})$ are linear decorated bicategories. We write $W^*_{fact}$ and $W^*_{fin}$ for these decorated bicategories. In \cite{Bartels1} an internalization of $W^*_{fin}$ is constructed through functorial extensions, to \textbf{vN}$^{fin}$ of the Haagerup standard form construction and the Connes fusion operation construction. We write $BDH$ for this double category. In \cite{yo2} the author proves that $W^*_{fact}$ and thus $W^*_{fin}$ are saturated, i.e. $H^*Q_{W^*_{fact}}=W^*_{fact}$ and $H^*Q_{W^*_{fin}}=W^*_{fin}$. The exact relation between $BDH$ and $Q_{W^*_{fact}}$ is provided by the canonical projection. We have the following consequence of \ref{thmprojection1}.
	
	\begin{prop}\label{canonicalprojectionvnalgebras}
		$\gamma BDH$ is a double quotient of $Q_{W^*_{fact}}$ through $\pi^{BDH}$.
	\end{prop}
	
	\noindent The category of squares $BDH_1$ of $BDH$ is the category whose objects and morphisms are Hilbert bimodules over factors and finite index equivariant intertwiners, i.e. the morphisms of $BDH$ are the squares of the form:

	\begin{center}
		
		\begin{tikzpicture}
			\matrix(m)[matrix of math nodes, row sep=4em, column sep=4em,text height=1.5ex, text depth=0.25ex]
			{A&B\\
				A'&B'\\};
			\path[->,font=\scriptsize,>=angle 90]
			(m-1-1) edge node[above]{$H$} (m-1-2)
			edge node[left]{$f$} (m-2-1)
			(m-2-1) edge node[below]{$H'$} (m-2-2)			
			(m-1-2) edge node[right]{$g$} (m-2-2)
			(m-1-1) edge [white] node [black][fill=white]{$(f,\varphi,g)$}(m-2-2) 
			
			;
		\end{tikzpicture}
	\end{center}

	\noindent where $A,A',B,B'$ are factors, $H$ is a left-right $A,B$-Hilbert bimodule, $H'$ is a left-right $A',B'$-Hilbert bimodule, $f:A\to A'$ and $g:B\to B'$ are unital $^*$-morphisms satisfying the inequalities
	
	\[[f(A),A']<\infty \ \mbox{and} \ [g(B),B']<\infty\]

	\noindent and $\varphi$ is a bounded operator from $H$ to $K$ satisfying the equation
	
	\[\varphi(a\xi b)=f(a)\varphi(\xi)g(b)\]

	\noindent for every $\xi\in H$ and $a\in A,b\in B$. In \cite{yo1} the category of squares of $\gamma BDH_1$ of $\gamma BDH$ was computed as the category of 2-subcyclic equivariant intertwiners. The fact that $BDH$ is a tensor double category means that there exists a tensor functor
	
	\[L^2:\mbox{\textbf{vN}}^{fin}\to BDH_1\]

	\noindent associating to every factor $A$ the Haagerup standar form $L^2(A)$ of $A$, and a tensor functor
	
	\[\boxtimes_\bullet:BDH_1\times_{\mbox{\textbf{vN}}^{fin}}BDH_1\to BDH_1\]

	\noindent associating to every compatible pair of squares $_AH_B,_BK_A$ its Connes fusion $_AH\boxtimes_B K_C$. The fact that these functors are compatible is expressed by the fact that $BDH$ is a tensor double category. The fact that these are operations on Hilbert bimodules and finite equivariant intertwiners is expressed by the equation

	\[H^*BDH=W^*_{fin}\]

	\noindent This equation is minimized by $\gamma BDH$ and the image category of the $L^2$-functor above is in $\gamma BDH_1$. We are thus interested in extending the functors
	
	\[L^2:\mbox{\textbf{vN}}^{fin}\to \gamma BDH_1\]

	\noindent and 
	
	\[\boxtimes_\bullet:\gamma BDH_1\times_{\mbox{\textbf{vN}}^{fin}}\gamma BDH_1\to \gamma BDH_1\]

	\noindent to compatible functors on \textbf{vN}$^{fact}$. The following proposition does this.

	\begin{prop}
		There exists a linear double category $\tilde{BDH}$ such that $H^*\tilde{BDH}=W^*_{fact}$ and such that $\gamma BDH$ is a sub-double category of $\tilde{BDH}$ satisfying the following condition: Given vertically or horizontally compatible squares $\varphi,\psi$ in $\tilde{BDH}$ if the vertical or horizontal composition of $\varphi$ and $\psi$ is in $BDH$ so are $\varphi$ and $\psi$.
	\end{prop}

	\begin{proof}
		We wish to prove that there exists a linear double category $\tilde{BDH}$ satisfying the equation $H^*\tilde{BDH}=W^*_{fact}$ and having $\gamma BDH$ as sub-double category in such a way that given every pair of squares $\varphi,\psi$ in $\tilde{BDH}$ such that either the vertical or the horizontal composition of $\varphi$ and $\psi$ is a square in $BDH$ then both $\varphi$ and $\psi$ are in $BDH$.
		
		Write $R$ for the equivalence relation on the collection of squares of $Q_{W^*_{fact}}$ defined as: $\varphi R\psi$ if $\pi^{\gamma BDH}\varphi=\pi^{\gamma BDH}\psi$. Thus defined $R$ is compatible with the vertical and horizontal structure data of $Q_{W^*_{fact}}$ and $\gamma BDH$ and thus $Q_{W^*_{fact}}/R$ is a globularly generated double category. We make $\tilde{BDH}$ to be this double category. From the equation $H^*Q_{W^*_{fact}}=W^*_{fact}$ the equation $H^*\tilde{BDH}=W^*_{fact}$ follows. $Q_{W^*_{fin}}$ is a sub-double category of $Q_{W^*_{fact}}$. Moreover, it is easily seen that the equtation

		\[Q_{W^*_{fin}}=\pi^{\gamma BDH-1}(\gamma BDH)\]

		\noindent holds. We thus obtain an isomorphism of double categories

		\[\pi^{\gamma BDH}\restriction_{Q_{W^*_{fin}}}\cong \gamma BDH\]

		\noindent We make use of the above isomorphism to identify $\gamma BDH$ with a sub-double category of $\tilde{BDH}$. The fact that pairs of squares $\varphi,\psi$ in $\gamma BDH$ satisfy the required condition inside $\tilde{BDH}$ follows by an easy induction argument on $min\left\{\ell\varphi,\ell\psi\right\}$ using the fact that given morphisms $f:A\to B$ and $g:B\to C$ in \textbf{vN}$^{fact}$ such that $[gf(A),C]<\infty$ then $[f(A),B],[g(B),C]<\infty$. This concludes the proof.
	\end{proof}
	
	\noindent The horizontal identity functor and the horizontal composition functor of $\tilde{BDH}$ are functors:

	\[L^2:\mbox{\textbf{vN}}^{fact}\to \tilde{BDH}_1\]

	\noindent and

	\[\boxtimes_\bullet:\tilde{BDH}_1\times_{\mbox{\textbf{vN}}^{fact}} \tilde{BDH}_1\to\tilde{BDH_1}\]
	
	\noindent compatible in the sense that $\tilde{BDH}$ is a double category and such that they restrict to the corresponding functors on $\gamma BDH$. By \cite[Proposition 6.1]{yo2} and Theorem \ref{thmprojection1} the space of morphisms of $\tilde{BDH}_1$ is the complex vector space spanned by formal vertical compositions of the form:

	\begin{center}
		
		\begin{tikzpicture}
			\matrix(m)[matrix of math nodes, row sep=4em, column sep=4em,text height=1.5ex, text depth=0.25ex]
			{A&A\\
				A&A\\
				B&B\\
				B&B\\};
			\path[->,font=\scriptsize,>=angle 90]
			(m-1-1) edge node [above]{$H$} (m-1-2)
			edge [blue]node{} (m-2-1)
			(m-2-1) edge [red]node{} (m-2-2)			
			(m-1-2) edge [blue]node{} (m-2-2)
			(m-1-1) edge [white] node [black][fill=white]{$\varphi$}(m-2-2)

			(m-2-1) edge node [left]{$f$} (m-3-1)
			(m-2-2) edge node [right]{$f$} (m-3-2)			
			(m-3-1) edge [red]node{} (m-3-2)
			(m-2-1) edge[white] node [black][fill=white]{$i_f$}(m-3-2)

			(m-3-1) edge [blue]node {} (m-4-1)
			(m-3-2) edge [blue]node {} (m-4-2)			
			(m-4-1) edge node[below]{$K$} (m-4-2)
			(m-3-1) edge[white] node [black][fill=white]{$\psi$}(m-4-2)
			;
		\end{tikzpicture}
	\end{center}

	\noindent where $A,B$ are factors, $f$ is a unital $^*$-morphism of possibly infinite index, $H$ is a left-right $A$ Hilbert bimodule, $K$ is a left-right $B$ Hilbert bimodule, $\varphi,\psi$ are bounded intertwiners from $H$ to $L^2(A)$ and from $L^2(B)$ to $K$ respectively, and $L^2(f)$ is a formal object in $\tilde{BDH}_1$. Whenever $f$ satisfies the inequality:
	
	\[[f(A),B]<\infty\]

	\noindent the formal symbol $L^2(f)$ is the image of $f$ under the $L^2$ functor of \cite{Bartels1}, the three term composition above is the corresponding composition in $BDH$. Moreover, this three term formal composition is a square in $\gamma BDH$ if and only if $[f(A),B]<\infty$. 
	
	In the construction presented in Proposition \ref{canonicalprojectionvnalgebras} we have not addressed the fact that we wish for $\tilde{BDH}$ to inherit, from $\gamma BDH$ the structure of a symmetric tensor double category. We will address this issue elsewhere.

	\

	\noindent \textit{Representability}

	\
	
	\noindent In Proposition \ref{canonicalprojectionvnalgebras} we have obtained a linear double category $\tilde{BDH}$ satisfying the equation 
	
	\[H^*\tilde{BDH}=W^*_{fact}\]
	
	\noindent and having $\gamma BDH$ as sub-double category. This provides compatible linear functors of the Haagerup standard form and the Connes fusion operation on linear categories of Hilbert bimodules and formal equivariant bounded intertwiners. We would like to obtain such functors, not on formal equivariant intertwiners, but on the category of Hilbert bimodules and actual equivariant intertwiners. We are not able to do this at the moment but Theorem \ref{thmprojection1} provides a possible solution to this. Assuming such functorial extensions exist, compatibility would provide a linear double category $C$ satisfying the equation
	
	\[H^*C=W^*_{fact}\]
	
	\noindent having $BDH$ as a sub-double category and such that the category of squares $C_1$ is a linear sub-category of the category \textbf{Mod}$^{fact}$ of Hilbert bimodules and equivariant intertwiners. Such category would be in the $\gamma$-fiber of a globularly generated double category, $\gamma C$, satisfying the equation
	
	\[H^*\gamma C=W^*_{fact}\]

	\noindent having $\gamma BDH$ as a sub-double category and such that $\gamma C_1$ is a linear sub-category of \textbf{Mod}$^{fact}$. In that case the morphism functor $\pi^{\gamma C}_1$ of $\pi^{\gamma C}$ will be a linear functor from $Q^\mathbb{C}_{W^*_{fact 1}}$ to \textbf{Mod}$^{fact}$ satisfying invariance conditions with respect to the double category structures of $Q^\mathbb{C}_{W^*_{fact}}$ and $BDH$. This suggests we should study the structure of the 2-category 
	
	\[Fun(Q^\mathbb{C}_{W^*_{fact 1}},\mbox{\textbf{Mod}}^{fact})\]

	\noindent under a possible set of initial conditions. We will analyze this point of view elsewhere, but we would like to obtain a categorical version of the above comments. In order to do this we need a way to understand functors between free globularly generated double categories. In the next section we study free double functors.
	
	\

	\section{Free double functors}\label{sFreeDoubleFunctors}\label{s4}

	\noindent In this section we introduce free double functors between free globularly generated double categories. We will use the free double functor construction to extend the free globularly generated double category construction to a functor. We use this construction to prove the results of Section \ref{sFGGCatasafreeobject}. The methods employed in the construction of free double functors mimic the construction of the canonical double projection of Section \ref{s3}. 
	
	\

	\noindent \textit{Strategy}
	
	\

	\noindent Given a pseudofunctor $G:\mathcal{B}\to\mathcal{B}'$ between decorated bicategories $\mathcal{B},\mathcal{B}'$ the free double functor $Q_G$ associated to $G$ will be a double functor from $Q_\mathcal{B}$ to $Q_{\mathcal{B}'}$ satisfying the equation:

	\[H^*Q_G\restriction_{\mathcal{B}}=G\]

	\noindent The strategy for the construction of $Q_G$ will be analogous to that of the construction of the canonical double projection of Section \ref{s3}. We first define $Q_G$ on formal squares of the form:

	
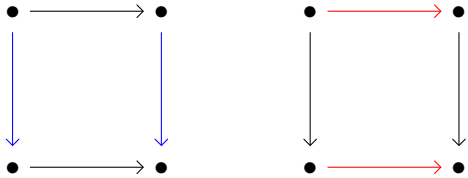
\begin{center}
		
		\begin{tikzpicture}
			\matrix(m)[matrix of math nodes, row sep=4em, column sep=4em,text height=1.5ex, text depth=0.25ex]
			{\bullet&\bullet&\bullet&\bullet\\
				\bullet&\bullet&\bullet&\bullet\\};
			\path[->,font=\scriptsize,>=angle 90]
			(m-1-1) edge node {} (m-1-2)
			edge [blue]node{} (m-2-1)
			(m-2-1) edge node{} (m-2-2)			
			(m-1-2) edge [blue]node{} (m-2-2)
			(m-1-1) edge [white] node [black][fill=white]{}(m-2-2) 
			
			(m-1-3) edge [red]node {} (m-1-4)
			edge node{} (m-2-3)
			(m-2-3) edge [red]node{} (m-2-4)			
			(m-1-4) edge node{} (m-2-4)
			
			;
		\end{tikzpicture}
	\end{center}

	\noindent The requirements in the definition of $Q_G$ force $Q_G$ to be uniquely defined by $G$ on the above squares. We freely extend this to a functor $F_1^G:F^\mathcal{B}_1\to F^{\mathcal{B}'}_1$. We extend this to a functor from $F_k^G:F_k^\mathcal{B}\to F_k^{\mathcal{B}'}$ inductively for all $k$ and we take the corresponding limit $F^G_\infty:F^\mathcal{B}_\infty\to F^{\mathcal{B}'}_\infty$. We prove that $F^G_\infty$ is compatible with both the structure data and the equivalence relations $R_\infty$ defining $Q_\mathcal{B}$ and $Q_{\mathcal{B}'}$ and that thus descends to the morphism functor of a double functor $Q_G$ from $Q_\mathcal{B}$ to $Q_{\mathcal{B'}}$. The coherence data for $Q_G$ will be inherited from that of $G$. Most of the technical results used in the construction of the free globularly generated double functor are analogous to arguments used in Section \ref{s3}. The precise statements are useful. We will thus record statements for these results but we will usually omit proofs.

	\

	\noindent \textit{Construction}

	\
	
	\noindent Let $\mathcal{B},\mathcal{B}'$ be decorated bicategories. Let $G:\mathcal{B}\to \mathcal{B}'$ be a decorated pseudofunctor. We begin the construction of $Q_G$ with the following lemma. Its proof is analogous to that of Proposition \ref{leminductionpi1} and we will omit it.

	\

	\begin{lem}\label{propconstructionfreefunctor}
		
		There exists a pair, formed by a sequence of functions $E_k^G:E_k^\mathcal{B}\to E_k^{\mathcal{B}'}$, and a sequence of functors $F_k^G:F_k^\mathcal{B}\to F_k^{\mathcal{B}'}$, with $k$ running over all positive integers, such that the following conditions are satisfied:
		
		\begin{enumerate}
			
			\item The equations $E_1^G\varphi=G\varphi$ and $E_1^G i_\alpha=i_\alpha$ hold for every 2-cell $\varphi$ in $\mathcal{B}$ and for every morphism $\alpha$ in $\mathcal{B}^*$.

			\item For every pair of positive integers $m,k$ such that $m\leq k$, the restriction of $E_k^G$ to the collection of morphisms of $F_m^\mathcal{B}$ is equal to the morphism function of $F_m^G$, and the restriction of the morphism function of $F_k^G$ to $E_m^\mathcal{B}$ is equal to $E_m^G$.

			\item The following two squares commute for every positive integer $k$:
			
			\begin{center}
				
				\begin{tikzpicture}
					\matrix(m)[matrix of math nodes, row sep=4em, column sep=4em,text height=1.5ex, text depth=0.25ex]
					{E_k^\mathcal{B}&E_k^{\mathcal{B}'}\\
						\mathcal{B}_1&\mathcal{B}'_1\\};
					\path[->,font=\scriptsize,>=angle 90]
					(m-1-1) edge node[auto] {$E_k^G$} (m-1-2)
					edge node[left] {$d_k^\mathcal{B},c_k^\mathcal{B}$} (m-2-1)
					(m-2-1) edge node[below]	{$G$} (m-2-2)			
					(m-1-2) edge node[auto] {$d_k^{\mathcal{B}'},c_k^{\mathcal{B}'}$} (m-2-2);
				\end{tikzpicture}

			\end{center}

			\item The following two squares commute for every positive integer $k$:
			
			\begin{center}
				
				\begin{tikzpicture}
					\matrix(m)[matrix of math nodes, row sep=4em, column sep=4em,text height=1.5ex, text depth=0.25ex]
					{E_k^\mathcal{B}&E_k^{\mathcal{B}'}\\
						\mbox{Hom}_{\mathcal{B}^*}&\mbox{Hom}_{\mathcal{B}'^*}\\};
					\path[->,font=\scriptsize,>=angle 90]
					(m-1-1) edge node[auto] {$E_k^G$} (m-1-2)
					edge node[left] {$s_k^\mathcal{B},t_k^\mathcal{B}$} (m-2-1)
					(m-2-1) edge node[below]	{$G^*$} (m-2-2)
					(m-1-2) edge node[right]{$s_k^\mathcal{B'},t_k^{\mathcal{B}'}$}(m-2-2);
				\end{tikzpicture}

			\end{center}

			\item The following two squares commute for every positive integer $k$:
			
			\begin{center}
				
				\begin{tikzpicture}
					\matrix(m)[matrix of math nodes, row sep=4em, column sep=4em,text height=1.5ex, text depth=0.25ex]
					{F_k^\mathcal{B}&F_k^{\mathcal{B}'}\\
						\mathcal{B}^*&\mathcal{B}'^*\\};
					\path[->,font=\scriptsize,>=angle 90]
					(m-1-1) edge node[auto] {$F_k^G$} (m-1-2)
					edge node[left] {$s_{k+1}^\mathcal{B},t_{k+1}^\mathcal{B}$} (m-2-1)
					(m-2-1) edge node[below]	{$G^*$} (m-2-2)			
					(m-1-2) edge node[auto] {$s_{k+1}^{\mathcal{B}'},t_{k+1}^{\mathcal{B}'}$} (m-2-2);
				\end{tikzpicture}

			\end{center}

			\item The following square commutes for every positive integer $k$
			
			\begin{center}
				
				\begin{tikzpicture}
					\matrix(m)[matrix of math nodes, row sep=4em, column sep=4em,text height=1.5ex, text depth=0.25ex]
					{E_k^\mathcal{B}\times_{\mbox{Hom}_{\mathcal{B}^*}} E_k^\mathcal{B}&E_k^{\mathcal{B}'}\times_{\mbox{Hom}_{\mathcal{B}'^*}} E_k^{\mathcal{B}'}\\
						E_k^\mathcal{B}&E_k^{\mathcal{B}'}\\};
					\path[->,font=\scriptsize,>=angle 90]
					(m-1-1) edge node[auto] {$E_k^G\times_G E_k^G$} (m-1-2)
					edge node[left] {$\ast_k^\mathcal{B}$} (m-2-1)
					(m-2-1) edge node[below]	{$E_k^G$} (m-2-2)			
					(m-1-2) edge node[auto] {$\ast_k^{\mathcal{B}'}$} (m-2-2);
				\end{tikzpicture}

			\end{center}

		\end{enumerate}
		
		\noindent Moreover, conditions 1-5 above determine the pair of sequences $E_k^G$ and $F_k^G$.

	\end{lem}

	\begin{obs} \label{obsfreedoublefunctors1}
		Let $m,k$ be positive integers such that $m\leq k$. Condition 2 of Proposition \ref{propconstructionfreefunctor} implies that the equations hold:
		
		\[E_k^G\restriction_{E_k^\mathcal{B}}=E_m^G \ \mbox{and} \ F_k^G\restriction_{F_m^\mathcal{B}}=F_m^G\]
		
	\end{obs}

	\begin{notation}\label{notationfreedoublefunctors}
		We will write $E_\infty^G$ for the limit $\varinjlim E_k^G$ in \textbf{Set} of the sequence $E_k^G$. Thus defined $E_\infty^G$ is a function from $E_\infty^\mathcal{B}$ to $E_\infty^{\mathcal{B}'}$. Further, we will write $F_\infty^G$ for the limit $\varinjlim F_k^G$ in \textbf{Cat} of the sequence of functors $F_k^G$. Thus defined, $F_\infty^G$ is a functor from $F_\infty^\mathcal{B}$ to $F_\infty^{\mathcal{B}'}$.
		
	\end{notation}

	\noindent The following observation follows directly from Lemma \ref{propconstructionfreefunctor} and Observation \ref{obsfreedoublefunctors1}.

	\begin{obs}\label{obsfreedoublefunctorsinfty}

		The pair $E_\infty^G,F_\infty^G$ satisfies the following conditions:
		
		\begin{enumerate}
			\item $E_\infty^G$ is equal to the morphism function of $F_\infty^G$.
			\item Let $k$ be a positive integer. The following equations hold:

			\[E_\infty^G\restriction_{E_k^\mathcal{B}}=E_k^G \ \mbox{and} \ F_\infty^G\restriction_{F_k^\mathcal{B}}=F_k^G\]
			
			\item The following squares commute:
			
			\begin{center}
				
				\begin{tikzpicture}
					\matrix(m)[matrix of math nodes, row sep=4em, column sep=4em,text height=1.5ex, text depth=0.25ex]
					{F_\infty^\mathcal{B}&F_\infty^{\mathcal{B}'}\\
						\mathcal{B}^*&\mathcal{B}'^*\\};
					\path[->,font=\scriptsize,>=angle 90]
					(m-1-1) edge node[auto] {$F_\infty^G$} (m-1-2)
					edge node[left] {$s_\infty^\mathcal{B},t_\infty^\mathcal{B}$} (m-2-1)
					(m-2-1) edge node[below]	{$G^*$} (m-2-2)			
					(m-1-2) edge node[auto] {$s_\infty^{\mathcal{B}'},t_\infty^{\mathcal{B}'}$} (m-2-2);
				\end{tikzpicture}

			\end{center}

			\item The following square commutes:
			
			\begin{center}
				
				\begin{tikzpicture}
					\matrix(m)[matrix of math nodes, row sep=4em, column sep=4em,text height=1.5ex, text depth=0.25ex]
					{E_\infty^\mathcal{B}\times_{\mbox{Hom}_{\mathcal{B}^*}} E_\infty^\mathcal{B}&E_\infty^{\mathcal{B}'}\times_{\mbox{Hom}_{\mathcal{B}'^*}} E_\infty^{\mathcal{B}'}\\
						E_\infty^\mathcal{B}&E_\infty^{\mathcal{B}'}\\};
					\path[->,font=\scriptsize,>=angle 90]
					(m-1-1) edge node[auto] {$E_\infty^G\times_G E_\infty^G$} (m-1-2)
					edge node[left] {$\ast_\infty^\mathcal{B}$} (m-2-1)
					(m-2-1) edge node[below]	{$E_\infty^G$} (m-2-2)			
					(m-1-2) edge node[auto] {$\ast_\infty^{\mathcal{B}'}$} (m-2-2);
				\end{tikzpicture}

			\end{center}
			
		\end{enumerate}
		
	\end{obs}

	\noindent It is easily seen, from the above observation, that the functor $F_\infty^G$ is compatible with the equivalence relations $R_\infty^\mathcal{B}$ and $R_\infty^{\mathcal{B}'}$. We will write $V_\infty^G$ for the functor from $V_\infty^\mathcal{B}$ to $V_\infty^{\mathcal{B}'}$ induced by $F_\infty^G$ and the equivalence relations $R_\infty^\mathcal{B}$ and $R_\infty^{\mathcal{B}'}$. We will write $H_\infty^G$ for the morphism function of $V_\infty^G$. Thus defined $H_\infty^G$ is function from $H_\infty^\mathcal{B}$ to $H_\infty^{\mathcal{B}'}$ induced by the function $E_\infty^G$ and the equivalence relations $R_\infty^\mathcal{B}$ and $R_\infty^{\mathcal{B}'}$. The following proposition follows directly from Observation \ref{obsfreedoublefunctorsinfty}.

	\begin{prop}\label{propfreedoublefunctor}
		$V_\infty^G$ satisfies the following conditions:
		
		\begin{enumerate}
			
			\item The following squares commute:
			
			\begin{center}
				
				\begin{tikzpicture}
					\matrix(m)[matrix of math nodes, row sep=4em, column sep=4em,text height=1.5ex, text depth=0.25ex]
					{V_\infty^\mathcal{B}&V_\infty^{\mathcal{B}'}\\
						\mathcal{B}^*&\mathcal{B}'^*\\};
					\path[->,font=\scriptsize,>=angle 90]
					(m-1-1) edge node[auto] {$V_\infty^G$} (m-1-2)
					edge node[left] {$s_\infty^\mathcal{B},t_\infty^\mathcal{B}$} (m-2-1)
					(m-2-1) edge node[below]	{$G^*$} (m-2-2)			
					(m-1-2) edge node[auto] {$s_\infty^{\mathcal{B}'},t_\infty^{\mathcal{B}'}$} (m-2-2);
				\end{tikzpicture}

			\end{center}
			
			\item The following square commutes
			
			\begin{center}
				
				\begin{tikzpicture}
					\matrix(m)[matrix of math nodes, row sep=4em, column sep=4em,text height=1.5ex, text depth=0.25ex]
					{V_\infty^\mathcal{B}\times_{\mathcal{B}^*} V_\infty^\mathcal{B}&V_\infty^{\mathcal{B}'}\times_{\mathcal{B}'^*} V_\infty^{\mathcal{B}'}\\
						V_\infty^\mathcal{B}&V_\infty^{\mathcal{B}'}\\};
					\path[->,font=\scriptsize,>=angle 90]
					(m-1-1) edge node[auto] {$V_\infty^G\times_G H_\infty^G$} (m-1-2)
					edge node[left] {$\ast_\infty^\mathcal{B}$} (m-2-1)
					(m-2-1) edge node[below]	{$V_\infty^G$} (m-2-2)			
					(m-1-2) edge node[auto] {$\ast_\infty^{\mathcal{B}'}$} (m-2-2);
				\end{tikzpicture}

			\end{center}

		\end{enumerate}
		
	\end{prop}

	\begin{notation}
		Let $\mathcal{B},\mathcal{B}'$ be decorated bicategories. Let $G:\mathcal{B}\to\mathcal{B}'$. We write $Q_G$ for the pair $(G^*,V_\infty^G)$.
	\end{notation}

	\noindent The following is the main theorem of this section.

	\begin{theorem}\label{functorialextensionthm}
		Let $\mathcal{B},\mathcal{B}'$ be decorated bicategories. Let $G:\mathcal{B}\to\mathcal{B}'$ be a decorated pseudofunctor. In that case $Q_G$ is the unique double functor from $Q_\mathcal{B}$ to $Q_{\mathcal{B}'}$ satisfying the equation:
		
		\[H^*Q_G\restriction_{\mathcal{B}}=G\]

	\end{theorem}

	\begin{proof}
		Let $\mathcal{B}$ and $\mathcal{B}'$ be decorated bicategories. Let $G:\mathcal{B}\to \mathcal{B}'$ be a decorated pseudofunctor. We wish to prove that in that case the pair $Q_G$ is a double functor from $Q_\mathcal{B}$ to $Q_{\mathcal{B}'}$ satisfying the equation 
		
		\[H^*Q_G\restriction{\mathcal{B}}=C\]
		
		\noindent A direct computation proves that the pair $Q_G$ intertwines the horizontal identity functors $i_\infty^\mathcal{B}$ and $i_\infty^{\mathcal{B}'}$ of $\mathcal{B}$. This, together with a direct application of Proposition \ref{propfreedoublefunctor} implies that the pair $Q_G$ is a double functor, with the coherence data of $G$ as coherence data. The object function of $V_\infty^G$ is equal to the restriction of $G$ to $\mathcal{B}_1$ and the restriction of $V_\infty^G$ to 2-cells of $\mathcal{B}$ is equal to the 2-cell function of $G$. This together with the fact that the object functor $Q_G$ is $G^*$ implies that the restriction to $\mathcal{B}$ of $H^*Q_G$ is equal to $G$. This concludes the proof. 
	\end{proof}

	\begin{definition}
		We call $Q_G$ above the free double functor associated to $G$. 
	\end{definition}

	\begin{obs}
		In the more general case in which $G$ is a lax/oplax decorated functor, the double functor $Q_G$ is also lax/oplax respectively.
	\end{obs}
	
	\

	\noindent \textit{Functoriality}

	\

	\noindent We now prove that the pair formed by the function associating $Q_\mathcal{B}$ to every decorated bicategory $\mathcal{B}$ and $Q_G$ to every decorated pseudofunctor $G$ is a functor from \textbf{bCat}$^*$ to \textbf{gCat}. We begin with the following lemma.

	\begin{lem}\label{functorialitylemma}
		Let $\mathcal{B},\mathcal{B}',\mathcal{B}''$ be decorated bicategories. Let $G:\mathcal{B}\to\mathcal{B}'$ and $G':\mathcal{B}'\to\mathcal{B}''$ be decorated pseudofunctors. The equations:

		\[H_k^{G'G}=H_k^{G'}H_k^G \ \mbox{and} \ V_k^{G'G}=V_k^{G'}V_k^G\]

		\noindent and

		\[H_k^{id_\mathcal{B}}=id_{H_k^{\mathcal{B}}} \ \mbox{and} \ V_k^{id_\mathcal{B}}=id_{V_k^\mathcal{B}}\]

		\noindent hold for every positive integer $k$.
	\end{lem}

	\begin{proof}
		Let $\mathcal{B},\mathcal{B}',\mathcal{B}''$ be decorated bicategories. Let $G:\mathcal{B}\to\mathcal{B}'$ and $G':\mathcal{B}'\to\mathcal{B}''$ be decorated pseudofunctors. Let $k$ be a positive integer. We wish to prove that $H_k^{G'G}=H_k^{G'}H_k^G$, that $V_k^{G'G}=V_k^{G'}V_k^G$, that $H_k^{id_\mathcal{B}}=id_{H_k^\mathcal{B}}$ and that $V_k^{id_\mathcal{B}}=id_{H_k^\mathcal{B}}$. We prove the first two of these equations. The proof of the remaining equations will be analogous.

		
		We proceed by induction on $k$. We first prove the equation $H_1^{G'G}=H_1^{G'}H_1^G$. Let $\varphi$ be a square in $\mathbb{G}^\mathcal{B}$. Suppose first that $\varphi$ is a 2-cell of $\mathcal{B}$. In that case $H_1^{G'G}\varphi=G'G\varphi$, which is equal to $G'\varphi G\varphi=H_1^{G'}H_1^G$. The equation clearly holds for horizontal identity squares in $Q_\mathcal{B}$. This proves the equation $H_1^{G'G}=H_1^{G'}H_1^G$ is true when restricted to $\mathbb{G}^\mathcal{B}$. This and the fact that $H_1^G,H_1^{G'}$ and $H_1^{G'G}$ satisfy condition 2 of Proposition \ref{propconstructionfreefunctor} proves that the equality extends to $H_1^\mathcal{B}$. Now $V_1^{G'G}$ and $V_1^{G'}V_1^G$ are equal to $H_1^{G'G}$ when restricted to $H_1^\mathcal{B}$. This and the fact that $H_1^\mathcal{B}$ generates $V_1^\mathcal{B}$ proves the equation $V_1^{G'G}=V_1^{G'}V_1^G$.
		
		Let now $k$ be a positive integer such that $k>1$. Suppose that for every $m<k$ the equations $H_m^{G'G}=H_m^{G'}H_m^G$ and $V_m^{G'G}=V_m^{G'}V_m^G$ hold. We now prove that the equations $H_k^{G'G}=H_k^{G'}H_k^G$ and $V_k^{G'G}=V_k^{G'}V_k^G$ hold. We first prove $H_k^{G'G}=H_k^{G'}H_k^G$. Both function $H_k^{G'G}$ and $H_k^{G'}H_k^G$ are equal to the morphism function of $V_{k-1}^{G'G}$ when restricted to the morphisms of $V_{k-1}^{\mathcal{B}}$. This, together with the way $H_k^{G'G}$ is defined, and the fact that $V_{k-1}^{G'G}$ satisfies condition 2 of Proposition \ref{propconstructionfreefunctor} implies that $H_k^{G'G}=H_k^{G'}H_k^G$ holds. Now, both $V_k^{G'G}$ and $V_k^{G'}V_k^G$ are equal to $H_k^{G'G}$ on the set of generators $H_k^\mathcal{B}$ of $V_k^\mathcal{B}$. This, together with the fact that $H_k^{G'G}$ satisfies condition 1 of Porposition \ref{propconstructionfreefunctor} implies the equation $V_k^{G'G}=V_k^{G'}V_k^G$ of $V_k^G$ and $V_k^{G'}$. The proof of the remaining two equations is analogous. This concludes the proof.    
	\end{proof}

	\begin{cor}\label{corfunctorialityfree}
		The pair formed by the function associating $Q_\mathcal{B}$ to every decorated bicategory $\mathcal{B}$ and $Q_G$ to every decorated pseudofunctor $G$ is a functor from \textbf{bCat}$^*$ to \textbf{gCat}.
	\end{cor}
	
	\begin{notation}
		We will write $Q$ for the functor defined in corollary \ref{corfunctorialityfree}. We will call $Q$ the free globularly generated double category functor.
	\end{notation}

	\noindent Let $k$ be a positive integer. If we write $H_k^\bullet$ for the pair formed by the function associating $H_k^\mathcal{B}$ to every decorated bicategory $\mathcal{B}$ and $H_k^G$ to every decorated pseudofunctor $G$ then $H_k^\bullet$ is a functor from \textbf{bCat}$^*$ to \textbf{Set} by Lemma \ref{functorialitylemma}. Similarly the pair $V_k^\bullet$ formed by the function associating $V_k^\mathcal{B}$ to every $\mathcal{B}$ and the functor $V_k^G$ to every decorated pseudofunctor $G$ is a functor from \textbf{bCat}$^*$ to \textbf{Cat}. Further, if we write $H^\bullet_\infty$ for the pair formed by the function associating $H_\infty^\mathcal{B}$ to every $\mathcal{B}$ and $H_\infty^G$ to every decorated bifunctor $G$ then $H_\infty^\bullet$ is a functor from \textbf{bCat}$^\ast$ to \textbf{Set} and if we write $V^\bullet_\infty$ for the pair formed by the function associating $V_\infty^\mathcal{B}$ to every decorated bicategory $\mathcal{B}$ and the function associating $V_\infty^G$ to every $G$ then $V^\bullet_\infty$ is a functor from \textbf{bCat}$^*$ to \textbf{Cat}. Thus defined $H_k^\bullet,H_\infty^\bullet$ relate by the equation $H_\infty^\bullet=\varinjlim V_k^\bullet$ and $V_k^\bullet,V_\infty^\bullet$ are related by the equation $V_\infty^\bullet=\varinjlim V_k^\bullet$.

	\section{Freeness}\label{sFGGCatasafreeobject}

	\noindent In this section we prove that the free globularly generated double category functor $Q$ defined in Section \ref{sFreeDoubleFunctors} is left adjoint to the restriction $H^*\restriction_{\mbox{\textbf{gCat}}}$ of $H^*$ to \textbf{gCat}, i.e. we prove the relation:

	\[Q\dashv H^*\restriction_{\mbox{\textbf{gCat}}}\]
	
	\noindent Further, we prove that $H^*\restriction_{\mbox{\textbf{gCat}}}$ is faithful thus making \textbf{gCat} into a concrete category over \textbf{bCat}$^*$ and $Q$ into a free functor on \textbf{gCat}. We interpret the results of this section by saying that the free globularly generated double category provides universal bases for $\gamma$-fibers with respect to $H^*$.
	
	\

	\noindent \textit{Adjoint relation}

	\

	\noindent We define a counit-unit pair for the adjuntion $Q\dashv H^*\restriction_{\mbox{\textbf{gCat}}}$. We will write $\pi^\bullet$ for the collection of canonical double projections $\pi^C$ with $C$ running over the objects of \textbf{gCat}. We prove the following proposition.

	\begin{prop}\label{naturalproposition}
		$\pi^\bullet$ is a natural transformation.
	\end{prop}
	
	\begin{proof}
		We wish to prove that $\pi^\bullet$ is a natural transformation from $H^*Q$ to identity $id_{\mbox{\textbf{gCat}}}$. That is, we wish to prove that for every double functor $T:C\to C'$ from a globularly generated double category $C$ to a globularly generated double category $C'$ the following square commutes:

		\begin{center}
			
			\begin{tikzpicture}
				\matrix(m)[matrix of math nodes, row sep=4em, column sep=4em,text height=1.5ex, text depth=0.25ex]
				{Q_{H^*C}&Q_{H^*C'}\\
					C&C'\\};
				\path[->,font=\scriptsize,>=angle 90]
				(m-1-1) edge node[auto] {$Q_{H^*T}$} (m-1-2)
				edge node[left] {$\pi^C$} (m-2-1)
				(m-2-1) edge node[below]	{$T$} (m-2-2)			
				(m-1-2) edge node[right] {$\pi^{C'}$} (m-2-2);
			\end{tikzpicture}

		\end{center}
		
		\noindent Let $C,C'$ be globularly generated double categories. Let $T:C\to C'$ be a double functor. We first prove that for each positive integer $k$ the following two squares commute:
		
		\begin{center}
			
			\begin{tikzpicture}
				\matrix(m)[matrix of math nodes, row sep=4em, column sep=4em,text height=1.5ex, text depth=0.25ex]
				{H_k^{H^*C}&H_k^{H^*C'}&V_k^{H^*C}&V_k^{H^*C'}\\
					\mbox{Hom}_{C_1}&\mbox{Hom}_{C'_1}&C_1&C'_1\\};
				\path[->,font=\scriptsize,>=angle 90]
				(m-1-1) edge node[auto] {$H_k^{H^*T}$} (m-1-2)
				edge node[left] {$H_k^{\pi^C}$} (m-2-1)
				(m-2-1) edge node[below]	{$T$} (m-2-2)			
				(m-1-2) edge node[auto] {$H_k^{\pi^{C'}}$} (m-2-2)
				(m-1-3) edge node[auto] {$V_k^{H^*T}$} (m-1-4)
				edge node[left] {$V_k^{\pi^{C}}$} (m-2-3)
				(m-1-4) edge node[auto]	{$V_k^{\pi^{C'}}$} (m-2-4)			
				(m-2-3) edge node[below] {$T$} (m-2-4);
			\end{tikzpicture}

		\end{center}
		
		\noindent We do this by induction on $k$. We fist prove that square on the left hand side commutes in the case $k=1$. Let $\varphi$ be a square in $\mathbb{G}^\mathcal{B}$. Suppose first that $\varphi$ is a 2-cell in $\mathcal{B}$. In that case $H_1^{H^*T}\varphi=H^*T\varphi$, which is equal to $T\varphi$. Now, $H_1^{\pi^C}\varphi=\varphi$ and thus the lower left corner of the left hand side square above is also equal to $T\varphi$. The square thus commutes in the values $k=1$ and $\varphi$ globular. An equally easy evaluation proves that the square also commutes for the values $k=1$ and $\varphi=i_f$ for any morphism $f$ in $\mathcal{B}^*$. We conclude that diagram on the left had side above commutes when restricted to collection $\mathbb{G}^\mathcal{B}$ in the case in which $k=1$. This together with the fact that $H_1^{H^*T}$ satisfies condition 6 of Proposition \ref{propconstructionfreefunctor} and the fact that $T$ is a double functor, implies that square commutes $H_1^\mathcal{B}$. Now, the square on the right hand side above restricts to the square on the left when restricted to the set of generators $E_1^\mathcal{B}$ of $V_1^\mathcal{B}$. This, together with the fact that all edges involved are functors implies that the square on the right commutes in the value $k=1$.
		
		Let now $k$ be a positive integer such that $k>1$. Suppose that for every positive $m<k$ the squares above commute. We now prove that the squares above commute for $k$. The square on the left hand side commutes when restricted to the collection of morphisms of $V_{k-1}^{H^*C}$ by the induction hypothesis. This, together  with the fact that the upper edge of the square satisfies condition 4 of Proposition \ref{propconstructionfreefunctor} and its left and right edges satisfy condition 5 of Proposition \ref{propconstructionfreefunctor} implies that the full square commutes. Now, the square on the right above is equal to the square on the left when restricted to the set $E_k^\mathcal{B}$ of generators of $V_k^\mathcal{B}$. This, together with the fact that all the edges of the square are functors, implies that the full square commutes on $k$. The result follows from this by taking limits.

	\end{proof}

	\noindent Let $\mathcal{B}$ be a decorated bicategory. In that case $\mathcal{B}$ is a sub-decorated bicategory of $H^*Q_\mathcal{B}$. We denote by $j^\mathcal{B}$ the inclusion of $\mathcal{B}$ in $H^*Q_\mathcal{B}$. We write $j$ for the collection of decorated pseudofunctors $j^\mathcal{B}$ with $\mathcal{B}$ running through the objects of \textbf{bCat}$^*$. As defined above $j$ is clearly a natural transformation from $id_{\mbox{\textbf{bCat}}^*}$ to $H^*Q$.

	\begin{theorem}\label{thmadjoint}
		$Q$ and $H^*\restriction_{\mbox{\textbf{gCat}}}$ satisfy the relation:
		
		\[Q\dashv H^*\restriction_{\mbox{\textbf{gCat}}}\]
		
		\noindent with the pair $(\pi^\bullet,j)$ as counit-unit pair.
	\end{theorem}

	\begin{proof}
		We wish to prove that pair formed by $Q$ and $H^*\restriction_{\mbox{\textbf{gCat}}}$ forms a left adjoint pair with $\pi^\bullet$ and $j$ as counit and unit respectively.
		
		It has already been established that $\pi^\bullet$ is a natural transformation from $QH^*$ to $id_{\textbf{gCat}}$, and that $j$ is a natural transformation from identity endofunctor $id_{\textbf{bCat}^*}$ of \textbf{bCat}$^*$ to $H^*Q$. We thus only need to prove that the pair $(\pi^\bullet,j)$ satisfies the triangle equations for a counit-unit pair. We begin by proving that the following triangle commutes:

		\begin{center}
			
			\begin{tikzpicture}
				\matrix(m)[matrix of math nodes, row sep=4em, column sep=4em,text height=1.5ex, text depth=0.25ex]
				{H^*&H^*QH^*\\
					&H^*\\};
				\path[->,font=\scriptsize,>=angle 90]
				(m-1-1) edge node[auto] {$j H^*$} (m-1-2)
				edge node[below] {$id_{H^*}$} (m-2-2)
				(m-1-2) edge node[auto] {$H^*\pi^\bullet$} (m-2-2);
			\end{tikzpicture}
			
		\end{center}

		\noindent Let $C$ be a globularly generated double category. The decoration and the collection of 1-cells of both $H^*C$ and $H^*Q_{H^*C}$ are equal to $C_0$ and to the collection of horizontal morphisms of $C$ respectively. Moreover, the restriction of $j_{H^*C}$ to both $C_0$ and the collection of horizontal morphisms of $C$ is the identity. The restriction of $j_{H^*C}$ to the collection of 2-cells of $H^*C$ is the inclusion of the collection of globular squares of $C$ to the collection of globular squares of $Q_C$. Now, again the restriction to both the decoration and the collection of horizontal morphisms of $H^*\pi^C$ is equal to the identity in both cases. The restriction of $H^*\pi^C$ to the collection of 2-cells of $H^*Q_{H^*C}$ is equal to the restriction, to the collection of globular squares of $C$, of $\pi^C$, which in turn is equal to the identity. It follows that $H^*\pi^C j_{H^*C}$ is equal to $id_{H^*C}$. We conclude that triangle above commutes.
		
		We now prove that the following triangle is commutative:
		
		\begin{center}
			
			\begin{tikzpicture}
				\matrix(m)[matrix of math nodes, row sep=4em, column sep=4em,text height=1.5ex, text depth=0.25ex]
				{Q&QH^*Q\\
					&Q\\};
				\path[->,font=\scriptsize,>=angle 90]
				(m-1-1) edge node[auto] {$Qj$} (m-1-2)
				edge node[left] {$id_Q$} (m-2-2)
				(m-1-2) edge node[auto] {$\pi^\bullet Q$} (m-2-2);
			\end{tikzpicture}
			
		\end{center}
		
		\noindent Let $\mathcal{B}$ be a decorated bicategory. In this case the restriction to both $C_0$ and to the collection of horizontal morphisms of $C$, of $j_\mathcal{B}$ in $H^*Q_\mathcal{B}$ is equal to the identity. The restriction, to both $C_0$ and to the collection of horizontal morphisms of $C$, now of $\pi_{Q_\mathcal{B}}$, is again equal to the identity. We now prove that the restriction, to $\mathcal{B}$, of $H^*\pi^{Q_\mathcal{B}}Q_{j_\mathcal{B}}$ of composition $\pi^{Q_\mathcal{B}}Q_{j_\mathcal{B}}$ defines a decorated endopseudofunctor of $\mathcal{B}$. It has already been established that the restriction, to both the decoration and the collection of horizontal morphisms of $\mathcal{B}$, of both $Q_{j_\mathcal{B}}$ and $\pi^{Q_\mathcal{B}}$ and thus of $\pi^{Q_\mathcal{B}}Q_{j_\mathcal{B}}$ is equal to the identity. Now, let $\varphi$ be a 2-cell in $\mathcal{B}$. In that case $Q_{j_\mathcal{B}}\varphi$ is equal to $j_\mathcal{B}\varphi$, which is equal to $\mathcal{B}$. Now, $\pi^{Q_\mathcal{B}}\varphi$ is again equal to $\varphi$. We conclude that the restriction to $\mathcal{B}$ of $H^*\pi^{Q_\mathcal{B}}Q_{j_\mathcal{B}}$ defines a decorated endopseudofunctor of $\mathcal{B}$. Moreover, this decorated endopseudofunctor of $\mathcal{B}$ is the identity endopseudofunctor of $\mathcal{B}$. It follows, from this, and from Proposition \ref{naturalproposition} that the composition $\pi^{Q_\mathcal{B}}Q_{j_\mathcal{B}}$ is equal to the identity endopseudofunctor of $Q_\mathcal{B}$. We conclude that triangle above commutes. This concludes the proof.
		
	\end{proof}

	\noindent As explained in the introduction we interpret Theorem \ref{thmadjoint} as a generalization of \cite[Theorem 5.3]{BrownMosa} and as a way to complete the diagram

	
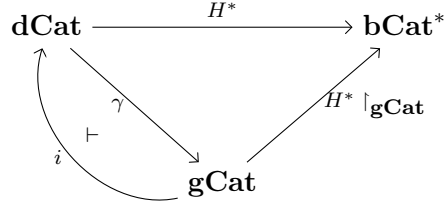
\begin{center}
		
		\begin{tikzpicture}
			\matrix(m)[matrix of math nodes, row sep=4em, column sep=3em,text height=1.5ex, text depth=0.25ex]
			{\mbox{\textbf{dCat}}&&\mbox{\textbf{bCat}}^*\\
				&\mbox{\textbf{gCat}}&\\};
			\path[->,font=\scriptsize,>=angle 90]
			(m-1-1) edge node[above] {$H^*$} (m-1-3)
			edge node[left] {$\gamma$} (m-2-2)			
			(m-2-2) edge node[right] {$H^*\restriction_{\mbox{\textbf{gCat}}}$} (m-1-3)
			(m-2-2) edge [ bend left=60] node [black][left]{$i$}(m-1-1)
			
			(m-2-2) edge [white, bend left=35] node [black][fill=white]{$\vdash$}(m-1-1);
		\end{tikzpicture}

	\end{center}

	\noindent to a diagram of the form:

	
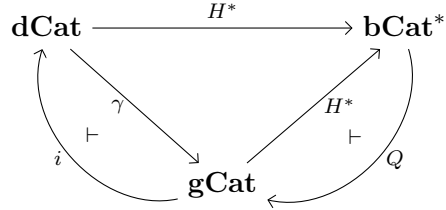
\begin{center}
		
		\begin{tikzpicture}
			\matrix(m)[matrix of math nodes, row sep=4em, column sep=3em,text height=1.5ex, text depth=0.25ex]
			{\mbox{\textbf{dCat}}&&\mbox{\textbf{bCat}}^*\\
				&\mbox{\textbf{gCat}}&\\};
			\path[->,font=\scriptsize,>=angle 90]
			(m-1-1) edge node[above] {$H^*$} (m-1-3)
			edge node[left] {$\gamma$} (m-2-2)			
			(m-2-2) edge node[right] {$H^*$} (m-1-3)
			
			(m-2-2) edge [ bend left=60] node [black][left]{$i$}(m-1-1)
			
			(m-2-2) edge [white, bend left=35] node [black][fill=white]{$\vdash$}(m-1-1)

			(m-1-3) edge [ bend left=60] node [black][right]{$Q$}(m-2-2)
			
			(m-1-3) edge [white, bend left=35] node [black][fill=white]{$\vdash$}(m-2-2);
		\end{tikzpicture}

	\end{center}

	\noindent Moreover, if we write \textbf{bCat}$^*_{sat}$ for the full subcategory of \textbf{bCat}$^*$ generated by saturated bicategories and we write \textbf{gCat}$^{free}$ for the full subcategory of \textbf{gCat} generated by the image of the object function of $Q$, then Theorem \ref{thmadjoint} and \cite[Corollary 3.4]{yo2} say that $H^*$ and $Q$ establish an equivalence between \textbf{bCat}$^*_{sat}$ y \textbf{gCat}$^{free}$. That is, we obtain the diagram:

	
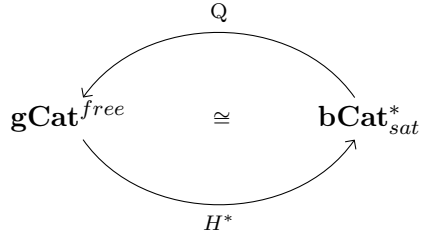
\begin{center}
		
		\begin{tikzpicture}
			\matrix(m)[matrix of math nodes, row sep=4em, column sep=6em,text height=1.5ex, text depth=0.25ex]
			{\mbox{\textbf{gCat}}^{free}&\mbox{\textbf{bCat}}^*_{sat}\\};
			\path[->,font=\scriptsize,>=angle 90]
			(m-1-1) edge [ bend right=55]node [black,below]{$H^*$} (m-1-2)
			
			(m-1-2) edge [bend right=55] node [black,above]{Q} (m-1-1)
			(m-1-1) edge [white]node[black,fill=white]{$\cong$} (m-1-2)

			;
		\end{tikzpicture}
	\end{center}

	\

	\noindent \textit{Faithfulness of decorated horizontalization}

	\

	\noindent We now prove that the decorated horizontalization functor $H^*$ is faithful when restricted to the category \textbf{gCat} of globularly generated double categories thus allowing an interpretation of \textbf{gCat} as a concrete category over \textbf{bCat}$^*$ and of $Q$ as a free construction. We begin with the following proposition.

	\begin{lem}\label{preliminaryfaithfulness}
		Let $C,C'$ be globularly generated double categories. Let $G,G':C\to C'$ be double functors. Suppose that the equation $H^*G=H^*G'$ holds. In that case the equations

		\[H^k_G=H^k_{G'} \ \mbox{and} \ V^k_G=V^k_{G'}\]
		
		\noindent hold for every $k\geq 1$.
	\end{lem}

	\begin{proof}
		Let $C,C'$ be globularly generated double categories. Let $G,G':C\to C'$ be double functors. Suppose that the equation $H^*G=H^*G'$ holds. Let $k$ be a positive integer. We wish to prove the equations $H_k^G=H_k^{G'}$ and $V_k^G=V_k^{G'}$ hold.
		
		We proceed by induction on $k$. We first prove the equation $H_1^G=H_1^{G'}$. Let $\varphi$ be a globular square in $C$. In that case $H_1^G=G\varphi$. This is equal, given that $\varphi$ is a globular square, to $H^*G\varphi$, which by the assumption of the lemma, is equal to $H^*G'\varphi$, and this is equal to $H_1^{G'}\varphi$. This, together with the fact that both $G$ and $G'$ are double functors and thus preserve horizontal identities implies that the functions $H_1^G$ and $H_1^{G'}$ are equal. We now prove the equation $V_1^G=V_1^{G'}$. Observe first that the restriction of the morphism function of $V_1^G$ to $H_1^C$ is equal to $H_1^G$ and that the restriction of the morphism function of $V_1^{G'}$ to $H_1^C$ is equal to $H_1^{G'}$. This, the previous argument, the fact that $H_1^C$ generates $V_1^C$, and the functoriality of $V_1^G$ and $V_1^{G'}$, implies that the equation $V_1^G=V_1^{G'}$ holds.
		
		Let now $k$ be a positive integer such that $k>1$. Suppose that for every $n<k$ the equations $H_n^G=H_n^{G'}$ and $V_n^G=V_n^{G'}$ hold. We prove that under these assumptions the equation $H_k^G=H_k^{G'}$ holds. Observe first that the restriction of $H_k^G$ to Hom$_{V_{k-1}^C}$ is equal to the morphism function of $V_{k-1}^G$ and that the restriction of $H_k^{G'}$ to Hom$_{V_{k-1}^C}$ is equal to the morphism function of $V_{k-1}^{G'}$. This, together with the induction hypothesis, and the fact that $G$ and $G'$ intertwine the source and target functors and the horizontal composition funtor of $C$ and $C'$ implies that $H_k^G=H_k^{G'}$. We now prove that $V_k^G$ and $V_k^{G'}$ are equal. Observe that the restriction of the morphism function of $V_k^G$ to $H_k^C$ is equal to $H_k^G$ and that the restriction of the morphism function of $V_k^{G'}$ to $H_k^C$ is equal to the function $H_k^{G'}$. This, together with the previous argument, the fact that $H_k^C$ generates $V_k^C$, and the functoriality of both $V_k^G$ and $V_k^{G'}$ proves that the functors $V_k^G$ and $V_k^{G'}$ are equal. This concludes the proof.
	\end{proof}

	\begin{cor}\label{corollaryfaithful1}
		Let $C,C'$ be double categories. Suppose $C$ is globularly generated. Let $G,G':C\to C'$ be double functors. If the equation $H^*G=H^*G'$ holds then the equation $G=G'$ also holds. 
	\end{cor}
	
	\begin{proof}
		Let $C,C'$ be double categories. Let $G,G':C\to C'$ be double functors. Suppose that $C$ is globularly generated and that the equation $H^*G=H^*G'$ holds. We wish to prove that the equation $G=G'$ holds.
		
		The globular pieces $\gamma G$ and $\gamma G'$  of $G$ and $G'$ are both  double functors from $C$ to $\gamma C'$. We have the equalities $H^*\gamma G=H^*G$ and $H^*\gamma G'=H^*G'$. It follows, from the assumption of the corollary that $\gamma G$ and $\gamma G'$ satisfy the assumptions of Proposition \ref{preliminaryfaithfulness} and thus the equation $V_k^{\gamma G}=V_k^{\gamma G'}$ are equal for everu $k$. Both $\gamma G$ and $\gamma G'$ admit decompositions as limits $\varinjlim V_k^{\gamma G}$ and $\varinjlim V_k^{\gamma G'}$. It follows, from this, that $\gamma G=\gamma G'$. Finally by the assumption that $C$ is globularly generated $\gamma G$ and $\gamma G'$ are equal to the codomain restrictions, from $C'$ to $\gamma C'$, of $G$ and $ G'$ respectively and thus $\gamma G$ and $\gamma G'$ are equal if and only if $G$ and $G'$ are equal. This concludes the proof. 
	\end{proof}

	\begin{prop}\label{Q faithful}
		$H^*\restriction_{\mbox{\textbf{gCat}}}$ is faithful.
	\end{prop}

	\noindent Proposition \ref{Q faithful} allows us to interpret \textbf{gCat} as a concrete category over \textbf{bCat}$^*$ through $H^*\restriction_{\mbox{\textbf{gCat}}}$. From this and from Theorem \ref{thmadjoint} we have the following corollary.

	\begin{cor}\label{Qfree}
		$Q$ is a free functor with respect to $H^*\restriction_{\mbox{\textbf{gCat}}}$.
	\end{cor}

	\noindent We interpret Corollary \ref{Qfree} by saying that the free globularly generated double category construction provides universal bases of fibers of the globularly generated piece fibration $\gamma$ and thus provides generators for globularly generated solutions to Problem \ref{prob}.
	
	\
	
	\noindent \textbf{Acknowledgements:} The author would like to thank the anonymous referee, whose comments and suggestions have greatly improved the paper. The author would also like to thank Robert Paré for his support, encouragement and interest in this project.


	
	\vspace{5mm}
	\noindent
	Juan Orendain \\
	Centro de Ciencias Matemáticas \\
	National University of Mexico \\
	Antigua Carretera a Pátzcuaro 8701 \\
	Residencial San José de la Huerta, 58089 Morelia, Mich. (Mexico) \\
	jorendain@matmor.unam.mx
	
	
\end{document}